\documentclass[11pt]{article}

  \usepackage[letterpaper, left=1in, right=1in, top=1in, bottom=1in]{geometry}
  
  \usepackage[affil-it]{authblk}  
  \providecommand{\keywords}[1]{\textbf{Keywords: } #1}
 
\def\KEYWORDSname{{\it Key words\/}{\kern0.7pt}:\enskip}
  \usepackage{enumerate}
  \usepackage{palatino}  
  \usepackage{setspace}  
  \usepackage{color}  
  \usepackage[dvipsnames]{xcolor}  

	\usepackage[labelfont=bf]{caption}	

  \usepackage{amsmath}  
  \usepackage{amssymb}  
  \usepackage{amsthm}  
  \usepackage{bm}  
  \usepackage{bbm}  
  \allowdisplaybreaks[4]  
	\usepackage{relsize}  
	\usepackage{algorithm}  
	\usepackage{algpseudocode}  

  \DeclareMathOperator*{\minimize}{minimize}
  \DeclareMathOperator*{\maximize}{maximize}
  \DeclareMathOperator*{\argmin}{arg\,min}
  \newcommand{\st}{\mathrm{subject\;to}}

	\theoremstyle{plain}
	\newtheorem{theorem}{Theorem}
	\newtheorem{proposition}{Proposition}
	\newtheorem{lemma}{Lemma}
	\newtheorem{corollary}{Corollary}
	\theoremstyle{definition}
	
	\newtheorem{example}{Example}
	
	\newtheorem{remark}{Remark}
	\newtheorem{condition}{Condition}

    \newtheorem{manuallemmainner}{Lemma}
    \newenvironment{manuallemma}[1]{%
        \renewcommand\themanuallemmainner{#1}%
        \manuallemmainner
    }{\endmanuallemmainner}
    
    \newtheorem{manualcorollaryinner}{Corollary}
    \newenvironment{manualcorollary}[1]{%
        \renewcommand\themanualcorollaryinner{#1}%
        \manualcorollaryinner
    }{\endmanualcorollaryinner}

    \newtheorem{manualpropositioninner}{Proposition}
    \newenvironment{manualproposition}[1]{%
        \renewcommand\themanualpropositioninner{#1}%
        \manualpropositioninner
    }{\endmanualpropositioninner}

  \usepackage{pifont}  

	\usepackage{longtable}  
	\usepackage{multirow}  
	\usepackage{array}  
	\usepackage{booktabs}  
	\usepackage{lscape} 

  \usepackage{graphicx}  
  \usepackage{float}  
  \usepackage{subfig}  
\usepackage{pgf,tikz,pgfplots}
\pgfplotsset{compat=1.15}
\usepackage{mathrsfs}
\usepackage{epsfig}
\usepackage{epstopdf}

	\usepackage{xr-hyper}
	\usepackage{hyperref}
	\hypersetup{
		pdfstartview = {FitV},  
		colorlinks = true,    
		linkcolor = NavyBlue,  
		citecolor = NavyBlue,  
		filecolor = NavyBlue,  
		urlcolor = NavyBlue  
	}  

  \usepackage[round]{natbib}  

\makeatletter
\newenvironment{varsubequations}[1]
 {%
  \addtocounter{equation}{-1}%
  \begin{subequations}
  \renewcommand{\theparentequation}{#1}%
  \def\@currentlabel{#1}%
 }
 {%
  \end{subequations}\ignorespacesafterend
 }
\makeatother

\title{Decomposition and Adaptive Sampling for \\ Data-Driven Inverse Linear Optimization}
\author[1]{Rishabh Gupta}
\author[1]{Qi Zhang \thanks{Corresponding author (qizh@umn.edu)}}
\affil[1]{Department of Chemical Engineering and Materials Science, University of Minnesota, Minneapolis, MN 55455, USA}
\date{}

\begin{document}

\maketitle

\begin{abstract}
\noindent This work addresses inverse linear optimization where the goal is to infer the unknown cost vector of a linear program. Specifically, we consider the data-driven setting in which the available data are noisy observations of optimal solutions that correspond to different instances of the linear program. We introduce a new formulation of the problem that, compared to other existing methods, allows the recovery of a less restrictive and generally more appropriate admissible set of cost estimates. It can be shown that this inverse optimization problem yields a finite number of solutions, and we develop an exact two-phase algorithm to determine all such solutions. Moreover, we propose an efficient decomposition algorithm to solve large instances of the problem. The algorithm extends naturally to an online learning environment where it can be used to provide quick updates of the cost estimate as new data becomes available over time. For the online setting, we further develop an effective adaptive sampling strategy that guides the selection of the next samples. The efficacy of the proposed methods is demonstrated in computational experiments involving two applications, customer preference learning and cost estimation for production planning. The results show significant reductions in computation and sampling efforts.
\end{abstract}


\section{Introduction}
Inverse optimization is an emerging new paradigm for uncovering hidden decision-making mechanisms from observed decision data. Following the principle of optimality \citep{Schoemaker1991} that is commonly applied in various fields including economics, psychology, and evolutionary biology, the key idea in inverse optimization is to model a decision-making process as an optimization problem. Decisions can then be viewed as optimal or near-optimal solutions of an optimization model, and the inverse optimization problem (IOP) is to infer this model, if otherwise unknown, from observations. A major advantage of this approach is its ability to explicitly include constraints. This eases the incorporation of domain knowledge, which is often readily available in the form of constraints, significantly. As a result, compared to common black-box machine learning methods, inverse optimization offers the promise of models with enhanced prediction accuracy and interpretability.

The notion of inverse optimization was first introduced by \cite{Burton1992} who considered the problem of determining travel costs on a network as perceived by the users given the routes they have taken. This has inspired research on several inverse network optimization problems \citep{Yang1997, Zhang1998, Zhang1999, Liu2006}. Since then, inverse optimization has found application in a myriad of fields, such as radiation therapy planning \citep{Chan2014, Babier2018}, investment portfolio optimization \citep{Bertsimas2012}, electricity demand forecasting \citep{Gallego2018}, auction mechanism design \citep{Beil2003, Birge2017}, biological systems \citep{Burgard2003, Arechavaleta2008, Terekhov2010}, and optimal control \citep{Hempel2015, Westermann2020}.

Early works in inverse optimization focus on determining an objective function that makes the observed decisions, given the constraints of the problem, exactly optimal. In their seminal paper, \cite{Ahuja2001} present a generalized solution method for inverse optimization with linear forward optimization problems (FOPs). Some of the later works extend the theory to consider conic \citep{Iyengar2005, Zhang2010}, discrete \citep{Schaefer2009, Wang2009, Bulut2015}, and nonlinear \citep{Chow2014} FOPs. 

More recently, the research focus has shifted towards \textit{data-driven} inverse optimization in which we observe an agent's decisions in multiple instances, which can be viewed as instances of the same optimization problem that differ in their input parameter values \citep{Esfahani2018}. With data-driven inverse optimization, one has a much greater chance of learning an optimization model that has true predictive power with respect to future decisions in unseen instances. Here, the observations are generally considered to be \textit{noisy} with the following being the three key sources of the noise: (i) measurement errors, (ii) bounded rationality of the decision maker, and (iii) model specification mismatch \citep{Aswani2018, Esfahani2018}. The existing literature for this setting is limited to the case of convex FOPs. The main distinction among the various proposed formulations is in terms of the loss function employed to fit the data. Minimization of the slack required to make the noisy data satisfy an optimality condition is considered by \cite{Boyd2011}, \cite{Bertsimas2015}, and \cite{Esfahani2018}. However, \cite{Aswani2018} show that this kind of loss function can lead to statistically inconsistent estimates and propose to minimize the sum of some norm of residuals with respect to the decision variables. In the data-driven context, most existing works assume that the observations are available in a single batch, while more recent contributions also address online learning environments in which the observations are made sequentially \citep{Barmann2017,Dong2018,Shahmoradi2019}.

In this work, we consider data-driven inverse \textit{linear} optimization with noisy observations in both batch and online learning settings. Here, the goal is to estimate the unknown cost vector of a linear program (LP). Inverse linear optimization constitutes an important class of IOPs as many decision-making problems can be formulated or approximated as LPs. Although inverse linear optimization falls into the broader category of inverse convex optimization for which established solution methods exist, these more general methods often yield overly restricted sets of admissible cost estimates when applied to inverse linear optimization with noisy data (as discussed in detail in Section \ref{sec:Challenges}). This limitation is to a great extent shared by tailored approaches specifically designed to solve the IOP in the linear case more efficiently \citep{Babier2018, Chan2019}. We note that the majority of the inverse linear optimization literature does not consider the data-driven case but focuses on the single-instance setting (with possibly multiple noisy observations).

Our proposed framework is designed to recover the complete set of admissible solutions for the inverse linear optimization problem, while incorporating the notion of a reference cost vector that represents the user's prior belief, which facilitates the selection of an appropriate point estimate. Based on a polyhedral understanding of the problem, we propose a two-phase approach that separates the tasks of denoising the data and parameter estimation. To solve large instances of the IOP, we develop an exact decomposition algorithm, which processes the data sequentially and can hence also serve as an efficient update method in online inverse optimization. For the online setting, we further develop an adaptive sampling strategy that guides the selection of the next samples in an effort to reduce the amount of required data. While adaptive sampling is quite a mainstream idea in machine learning \citep{Domingo2002, Chang2005, Cozad2014}, to the best of our knowledge, it has not yet been considered in inverse optimization. We believe that the development of such a framework can go a long way in increasing the acceptance of inverse optimization as an alternative to black-box modeling methods, especially in situations where data acquisition is expensive or time-intensive.

The main contributions of this work are as follows:
\begin{enumerate}
    \item We introduce a new general formulation of the data-driven inverse linear optimization problem that considers multiple noisy observations collected for multiple experiments, which are problem instances with different input parameter values. We highlight several geometrical properties of the problem, and show that by assuming that optimal solutions lie at the vertices of the feasible region, we can recover the complete set of admissible cost estimates.
    \item We show that the proposed IOP formulation yields a finite number of solutions. We introduce a two-phase algorithm that can recover all such solutions. Furthermore, we show that under a very mild condition, the IOP is guaranteed to have a unique solution.
    \item We develop an efficient sequential decomposition algorithm to solve large instances of the IOP. The algorithm directly extends itself to online inverse optimization where it can be used to provide quick updates of the cost estimate as new data becomes available.
    \item We propose an effective adaptive sampling strategy that guides the choice of the experiments in online inverse optimization. The adaptive sampling problem is formulated as a mixed-integer nonlinear program (MINLP) for which we provide an efficient heuristic solution algorithm.
    \item We demonstrate the effectiveness of the proposed framework through a comprehensive set of computational experiments, addressing the problems of customer preference learning and cost estimation for production planning. The results indicate that generally, reasonable prediction accuracies can be achieved with relatively small numbers of experiments. Also, one can observe significant reductions in solution time and required number of samples due to the proposed decomposition and adaptive sampling methods, respectively.
\end{enumerate}

The remainder of this paper is organized as follows. In Section \ref{sec:Challenges}, we present a formal description of the inverse linear optimization problem. In Section \ref{sec:NewApproach}, we propose a new formulation that utilizes a reference cost vector to find reasonable cost estimates, discuss its properties, and develop a two-phase solution algorithm. Section \ref{sec:DecomAlg} introduces an exact decomposition algorithm that allows the efficient solution of large instances of the IOP and naturally extends to online inverse optimization. Our proposed adaptive sampling framework is detailed in Section \ref{sec:ASP}. In Section \ref{sec:CompSt}, results from the computational studies are presented. Finally, we conclude in Section \ref{sec:Conc}. All omitted proofs can be found in the online supplement, section A.

\section{Background of Inverse Linear Optimization}
\label{sec:Challenges}

Consider a decision-making problem that can be represented as an LP of the following form:
\begin{equation}
\label{eqn:FOP}
\tag{FOP}
\begin{aligned}
    \minimize\limits_{x \in \mathbb{R}^n} \quad & c^{\top} x \\
    \st \quad & A x \leq b,
\end{aligned}
\end{equation}
which we call the forward optimization problem (FOP). The cost vector $c \in \mathbb{R}^n$ is unknown; however, experiments perturbing values in $A \in \mathbb{R}^{m \times n}$ and $b \in \mathbb{R}^{m}$ can be designed to help improve our estimate of $c$. These experiments are subject to certain problem-specific restrictions on $A$ and $b$, and we denote the set of their allowed values by $\Pi$, i.e. $(A,b) \in \Pi$. Moreover, we assume that $\Pi$ is such that for any $(A,b) \in \Pi$, the polyhedron represented by $Ax \leq b$ is compact and nonempty. This is a mild assumption as in essentially all real-world problems, the decision variables are bounded. The results of the perturbation experiments, which are assumed to be optimal solutions to (FOP), are observed with some random noise. For a specific experiment, multiple samples can be collected such that one can recover the true optimal solution with some confidence. In what follows, we refer to this estimate of the true optimal solution as the \textit{denoised} estimate, and the process of obtaining it as \textit{denoising} the data.

Given observations, the inverse optimization problem (IOP) is to obtain an estimate of $c$, $\hat{c}$, such that the difference between the observations and the solutions obtained from solving \eqref{eqn:FOP} with $\hat{c}$ as the cost vector is minimized. The IOP is commonly formulated as follows:
\begin{equation}
\label{eqn:IOP_general}
\begin{aligned}
    \minimize\limits_{\hat{c} \in \mathbb{R}^n, \, \hat{x}} \quad & \sum\limits_{i \in \mathcal{I}}\sum\limits_{j \in \mathcal{J}_i}\, \lVert x_{ij} - \hat{x}_{ij} \rVert \\
    \st \quad & \hat{x}_{ij} \in \argmin\limits_{\tilde{x} \in \mathbb{R}^n} \left\lbrace \hat{c}^{\top} \tilde{x}:  A_i \tilde{x} \leq b_i \right\rbrace \quad \forall \, i \in \mathcal{I}, \, j \in \mathcal{J}_i,
\end{aligned}
\end{equation}
where $\mathcal{I}$ is the set of experiments, where each experiment $i$ is associated with inputs $(A_i, b_i)$, $\mathcal{J}_i$ denotes the set of noisy observations for experiment $i$, and $x_{ij}$ is the observed output for $j \in \mathcal{J}_i$. The objective is to choose $\hat{c}$ and $\hat{x}$ such that the loss function, which is defined as the sum of some norm of the residuals, is minimized. Formulation \eqref{eqn:IOP_general} generalizes existing variants of the IOP from the literature. Some consider a setting in which $A$ and $b$ cannot be changed, which leads to the case of $|\mathcal{I}| = 1$ \citep{Chan2018, Chan2019}. Others consider random sampling of $A$ and $b$ without assigning the samples to distinct sets corresponding to specific inputs \citep{Aswani2018}; in this case, we have $|\mathcal{I}| = N$ with $N$ being the total number of samples, and $|\mathcal{J}_i| = 1$ for all $i \in \mathcal{I}$. In fact, splitting the set of samples into input-specific subsets does not affect a formulation like \eqref{eqn:IOP_general}; however, it will be an essential feature of our proposed alternative approach (more in Section \ref{sec:NewApproach}).

Problem \eqref{eqn:IOP_general} is a bilevel optimization problem and is typically solved by replacing its lower-level problems with their optimality conditions. While most existing works make use of strong duality \citep{Aswani2018, Chan2019, Shahmoradi2019}, some have also applied reformulations based on the Karush-Kuhn-Tucker (KKT) conditions \citep{Boyd2011, Saez-Gallego2016}. In case of LPs, both optimality conditions are equivalent. In this work, we use a KKT-based approach as the duality-based formulation is nonlinear and nonconvex in the constraints due to the presence of bilinear terms, whereas the constraints of the KKT-based formulation can be linearized by introducing binary variables, which is advantageous from a computational standpoint. Thus, we arrive at the following mixed-integer reformulation of \eqref{eqn:IOP_general}:
\begin{subequations}
\label{eqn:IOP_general_KKT}
\begin{align}
    \minimize\limits_{\hat{c} \in \mathbb{R}^n, \, \hat{x}, \, s, \, \lambda, \, z} \quad & \sum\limits_{i \in \mathcal{I}}\sum\limits_{j \in \mathcal{J}_{i}} \lVert x_{ij} - \hat{x}_{ij} \rVert \\
    \st  \quad \; & \hat{c} + A_i^{\top} \lambda_{ij} = 0 \quad \forall \, i \in \mathcal{I},\, j \in \mathcal{J}_i \label{eqn:dualfeas} \\
    & A_i\hat{x}_{ij} + s_{ij} = b_i \quad \forall \, i \in \mathcal{I},\, j \in \mathcal{J}_{i} \label{eqn:primalfeas}\\
    & \lambda_{ij} \leq Mz_{ij} \quad \forall \, i \in \mathcal{I},\, j \in \mathcal{J}_{i} \label{eqn:complementarity1} \\
    & s_{ij} \leq M(e-z_{ij}) \quad \forall \, i \in \mathcal{I},\, j \in \mathcal{J}_{i} \label{eqn:complementarity2} \\
    & \hat{x}_{ij} \in \mathbb{R}^n, \, s_{ij} \in \mathbb{R}_+^m, \, \lambda_{ij} \in \mathbb{R}_+^m, \, z_{ij} \in \{0,1\}^m \quad \forall \, i \in \mathcal{I}, \, j \in \mathcal{J}_{i} \label{eqn:varbounds},
\end{align}
\end{subequations}
where $M$ is a sufficiently large parameter and $e$ denotes the all-ones vector. Constraints \eqref{eqn:dualfeas}, \eqref{eqn:primalfeas}, and \eqref{eqn:complementarity1}--\eqref{eqn:complementarity2} correspond to the stationarity, primal feasibility, and complementary slackness conditions, respectively. The following theorem characterizes the solution set of \eqref{eqn:IOP_general_KKT}.

\begin{theorem}\label{thm:IntersectCones}
For a given feasible $\hat{x}$, the set of feasible $\hat{c}$ in problem \eqref{eqn:IOP_general_KKT} is a polyhedral cone.
\end{theorem}
\begin{proof}
Consider \eqref{eqn:IOP_general_KKT} for a specific experiment $i$ and sample $j \in \mathcal{J}_i$ and a corresponding feasible $\hat{x}_{ij}$. Let $\mathcal{K}_i$ be the set of constraints for experiment $i$, i.e. resulting from $A_i$ and $b_i$. From \eqref{eqn:varbounds}, we have that $\lambda_{ij} \geq 0$; hence, if $\lambda_{ijk} = 0$ for all $k \in \mathcal{K}_i$, then \eqref{eqn:dualfeas} imply that $\hat{c} = 0$. Otherwise, if $\exists \, k \in \mathcal{K}_i$ such that $\lambda_{ijk} > 0$, then from constraints \eqref{eqn:complementarity1}--\eqref{eqn:complementarity2}, we have that $s_{ijk} = 0$, and \eqref{eqn:primalfeas} imply that $\hat{x}_{ij}$ is such that $a_{ik}^{\top}\hat{x}_{ij} = b_{ik}$, where $a_{ik} \in \mathbb{R}^n$ defines the $k$th row of $A_i$. Hence, from \eqref{eqn:dualfeas}, we have that $\hat{c} \in \mathrm{cone}\left(\{-a_{it}\}_{t \in \mathcal{T}(\hat{x}_{ij})}\right)$, where $\mathcal{T}(\hat{x}_{ij})$ denotes the set of constraints active at $\hat{x}_{ij}$. Since $\mathcal{T}(\hat{x}_{ij})$ is a finite set, the feasible region for $\hat{c}$ associated with experiment $i$ and sample $j$ is a polyhedral cone. As this statement holds for every $i \in \mathcal{I}$ and $j \in \mathcal{J}_i$, we have
\begin{equation*}
    \hat{c} \in \bigcap\limits_{i \in \mathcal{I}, \, j \in \mathcal{J}_i} \mathrm{cone}\left(\{-a_{it}\}_{t \in \mathcal{T}(\hat{x}_{ij})}\right),
\end{equation*} 
which is the intersection of a finite number of polyhedral cones, hence also a polyhedral cone.
\end{proof}

\paragraph{Admissible Set.}
As a consequence of Theorem \ref{thm:IntersectCones}, there is no unique solution to problem \eqref{eqn:IOP_general}, or equivalently \eqref{eqn:IOP_general_KKT}, since for any optimal $\hat{c}$, $\alpha \hat{c}$ with $\alpha$ being any nonnegative scalar different than 1 yields another optimal solution. Instead, by means of Theorem 1, we can determine the full set of ``inverse-optimal'' $\hat{c}$, which we refer to as the \textit{admissible set}. Also, to avoid the trivial solution $\hat{c} = 0$, it is common to use a slight variant of \eqref{eqn:IOP_general} that restricts the length of $\hat{c}$ by adding a norm constraint, e.g. $\lVert \hat{c} \rVert_p = 1$ \citep{Esfahani2018, Chan2019, Shahmoradi2019}. However, the use of a norm constraint introduces additional nonconvexity into the problem formulation. While the choice of $p$-norm has been arbitrary, under some special conditions on $c$, the $1$-norm and $\infty$-norm have been shown to lead to tractable formulations \citep{Chan2019}.

\paragraph{Noisy Observations.}
As \eqref{eqn:FOP} is an LP, with a nonzero $c$, any optimal solution will lie on the boundary of the polyhedral feasible region. Problem \eqref{eqn:IOP_general} with a $p$-norm constraint on $c$ can be interpreted as the projection of noisy observations onto one of the polyhedron's facets such that the total projection distance is minimized \citep{Chan2019}. While this approach provides good solutions when the FOP is strongly convex, it often leads to a severely restricted admissible set when the FOP is an LP. As illustrated in Figure \ref{fig:Feasreg}, even if the true solution lies at a vertex, noise in the data can cause them to get projected onto one of the facets, making a vector orthogonal to that facet the only feasible $\hat{c}$. As highlighted by \cite{Shahmoradi2019}, this formulation also leads to unstable predictions in the presence of outliers in the data.

\begin{figure}[h!]
\includegraphics[width=0.25\textwidth]{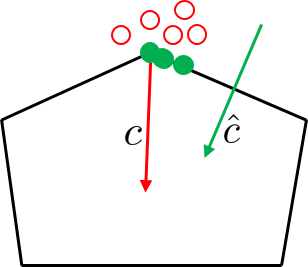}
\centering
\caption{(Color online) The polytope represents the feasible region of the FOP defined by $A_ix \leq b_i$, noisy data samples $x_{ij}$ are depicted by red hollow circles, and their denoised estimates $\hat{x}_{ij}$ are shown as green filled circles. The true cost vector is $c$, whereas $\hat{c}$ is the solution to \eqref{eqn:IOP_general} with an additional $p$-norm constraint on $\hat{c}$.}
\label{fig:Feasreg}
\end{figure}

\paragraph{Reference Cost.}
It is important to note that the problem of estimating $c$ given noisy observations consists of two tasks that have to be performed simultaneously: (i) denoising the data to obtain estimates of the real optimal solutions, and (ii) using these solutions to estimate $c$. Formulations of the form \eqref{eqn:IOP_general} and existing variants thereof mainly address the first aspect, but have deficiencies, as elucidated above, when it comes to finding a good point estimate of $c$, especially when the admissible set is large or the number of observations is small. Traditionally, when sampling is deterministic, inverse optimization has been facilitated by using a reference cost vector $\bar{c}$ and employing an objective function of the form $\lVert \bar{c} - \hat{c} \rVert_2^2$ \citep{Ahuja2001, Heuberger2004}. Such a reference cost represents a prior belief that is available in most practical applications, e.g. obtained through first-principles modeling or expert knowledge. A cost estimate $\hat{c}$ that is close to a reference $\bar{c}$ is often desired. It seems that the notion of such a reference, which can aid the process of recovering the real cost vector, has largely been ignored in the noisy case. 

\section{A Two-Phase Approach to Inverse Optimization}
\label{sec:NewApproach}

In this section, we propose a general inverse linear optimization model that finds a nontrivial estimate of the cost vector that, among all the ones that minimize the loss function, most closely resembles a reference cost vector. We discuss the main properties of this IOP and develop an exact two-phase solution algorithm.

\subsection{Problem Formulation and Properties}
We first define $\mathcal{C}'$ as follows:
\small
\begin{equation}
\label{eqn:Cprime_def}
    \mathcal{C}' := \left\lbrace \hat{c}: (\hat{c}, \hat{x}) \in \argmin\limits_{\hat{c} \in \mathbb{R}^n \setminus \{0\}, \, \hat{x}} \left\lbrace \sum\limits_{i \in \mathcal{I}} \sum\limits_{j \in \mathcal{J}_{i}} \, \lVert x_{ij} - \hat{x}_{ij} \rVert: \hat{x}_{ij} \in \argmin\limits_{\tilde{x} \in \mathbb{R}^n} \left\lbrace \hat{c}^{\top} \tilde{x}:  A_i \tilde{x} \leq b_i \right\rbrace \; \forall \, i \in \mathcal{I}, \, j \in \mathcal{J}_i \right\rbrace \right\rbrace,
\end{equation}
\normalsize
where it is assumed that for all $i \in \mathcal{I}$, $(A_i, b_i)$ is chosen from some set $\Pi$ such that the polyhedron $\{\tilde{x}: A_i \tilde{x} \leq b_i\}$ is compact and nonempty. 

Note that the loss minimization problem \eqref{eqn:Cprime_def} is essentially problem \eqref{eqn:IOP_general}; however, the above representation emphasizes the fact that \eqref{eqn:IOP_general} does not have a unique optimal $\hat{c}$, but rather a set of optimal solutions. The set $\mathcal{C}'$ consists of all these optimal $\hat{c}$ except for the trivial solution $\hat{c} = 0$. Next, we solve the following problem to choose a $\hat{c}$ from $\mathcal{C} = \mathcal{C}' \cup \{0\}$ that most resembles a known reference ${\bar{c}}$:
\begin{equation}
\label{eqn:InitialIOP}
\begin{aligned}
    \minimize\limits_{\hat{c}} \quad & \lVert \bar{c} - \hat{c} \rVert_2^2 \\
    \st \quad & \hat{c} \in \mathcal{C}.
\end{aligned}
\end{equation}
While we exclude the trivial solution $\hat{c} = 0$ from $\mathcal{C}'$, problem \eqref{eqn:InitialIOP} considers $\mathcal{C}$ as the set of admissible cost estimates, which does include the trivial solution. The rationale behind this setup is the following: When determining $\mathcal{C}'$, we do not want to consider $\hat{c} = 0$, which allows a minimum loss but does not provide any useful information. However, admitting the trivial solution in \eqref{eqn:InitialIOP} can be helpful since the unlikely case in which $\hat{c} = 0$ is the optimal solution to \eqref{eqn:InitialIOP} would immediately indicate that $\bar{c}$ is a very bad reference. In addition, as discussed later in this section, admitting $\mathcal{C}$ in \eqref{eqn:InitialIOP} results in useful theoretical properties.

As mentioned in Section \ref{sec:Challenges}, solving problem \eqref{eqn:InitialIOP} provides good results in the deterministic case; however, it fails when noisy data are considered since projection of data onto a facet causes $\mathcal{C}$ to be a single ray. To overcome this issue, we propose to consider, instead of $\mathcal{C}'$, the following slightly modified set:
\small
\begin{equation}
\label{eqn:Vertex_set}
    \widehat{\mathcal{C}}' := \left\lbrace \hat{c}: (\hat{c}, \hat{x}) \in \argmin\limits_{\hat{c} \in \mathbb{R}^n \setminus \{0\}, \, \hat{x}} \left\lbrace \sum\limits_{i \in \mathcal{I}} \sum\limits_{j \in \mathcal{J}_{i}} \, \lVert x_{ij} - \hat{x}_{i} \rVert: \hat{x}_{i} \in \argmin\limits_{\tilde{x} \in \mathbb{R}^n} \left\lbrace \hat{c}^{\top} \tilde{x}:  A_i \tilde{x} \leq b_i \right\rbrace \cap \mathcal{V}_i \; \forall \, i \in \mathcal{I} \right\rbrace \right\rbrace,
\end{equation}
\normalsize
where $\mathcal{V}_i$ denotes the set of extreme points of $\{\tilde{x}: A_i \tilde{x} \leq b_i\}$. This leads to the following IOP:
\begin{equation}
\label{eqn:IOP}
\tag{IOP}
\begin{aligned}
    \minimize\limits_{\hat{c}} \quad & \lVert \bar{c} - \hat{c} \rVert_2^2 \\
    \st \quad & \hat{c} \in \widehat{\mathcal{C}},
\end{aligned}
\end{equation}
where $\widehat{\mathcal{C}} = \widehat{\mathcal{C}}' \cup \{0\}$. The set $\widehat{\mathcal{C}}'$ considers the projection of the data for each experiment $i$ to one of the vertices of $\{\tilde{x}: A_i \tilde{x} \leq b_i\}$ such that the total projection distance is minimized. 
We consider two different cases depending on the nature of \eqref{eqn:FOP} to argue why this approach leads to a more appropriate admissible set. For the ease of exposition, our discussion is restricted to the case of a single experiment, i.e. $|\mathcal{I}| = 1$, with multiple samples.

\begin{enumerate}
\item If \eqref{eqn:FOP} has a unique optimal solution, the noisy samples are likely to be located close to the vertex at which the optimal solution lies (see Figure \ref{fig:vertexproja}). By solving the loss minimization problem in \eqref{eqn:Vertex_set}, we obtain a vertex $\hat{x}_i$, and the resulting $\widehat{\mathcal{C}}$ is an exhaustive set of cost vectors that can make $\hat{x}_i$ optimal for \eqref{eqn:FOP}. Recall from Theorem \ref{thm:IntersectCones} that any $\hat{c} \in \widehat{\mathcal{C}}$ can be expressed as a conic combination of the vectors orthogonal to the facets associated with the constraints active at $\hat{x}_i$. Thus, if $\hat{x}_i$ is the ``right" vertex, i.e. the true optimal solution, the true cost vector $c$ will be one of the vectors in $\widehat{\mathcal{C}}$. This is in contrast to $\mathcal{C}$, which in this case would be a single ray that does not contain $c$.

\item If \eqref{eqn:FOP} has multiple optimal solutions, the noisy samples are likely to be located close to the facet that represents the set of optimal solutions (see Figure \ref{fig:vertexprojb}). If solving the loss minimization problem in \eqref{eqn:Vertex_set} results in a vertex $\hat{x}_i$ that is an optimal solution to \eqref{eqn:FOP}, $c$ will be one of the extreme rays of $\widehat{\mathcal{C}}$. For this case, in most practical instances, we expect the data to show a preference toward a particular vertex of the facet, i.e. there is a unique optimal $\hat{x}_i$. Even if multiple $\hat{x}_i$ achieve the same loss, we still have $c \in \widehat{\mathcal{C}}$ as long as at least one of the $\hat{x}_i$ is in fact an optimal solution to \eqref{eqn:FOP}. In this case, one may argue that $\widehat{\mathcal{C}}$ is overly large compared to $\mathcal{C}$; however, if $\bar{c} = c$, the true cost vector will be recovered when solving \eqref{eqn:IOP}.
\end{enumerate}
 
\begin{figure}[h]
\centering
\subfloat[]{
    \label{fig:vertexproja}
    \includegraphics[width=0.25\textwidth]{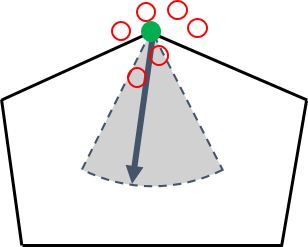}} \hspace{3em}
\subfloat[]{
    \label{fig:vertexprojb}
    \includegraphics[width=0.25\textwidth]{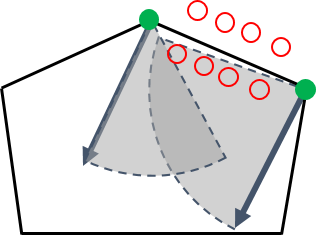}}
\caption{(Color online) Different cases when projection of noisy data is restricted to one of the vertices of the polyhedral feasible region. Arrows show the true cost vector $c$, noisy data samples $x_{ij}$ are depicted by red hollow circles, and their denoised estimates obtained from solving the loss minimization problem in \eqref{eqn:Vertex_set} are shown as green filled circles. The gray shaded regions indicate the set $\widehat{\mathcal{C}}$. In (b), the loss minimization problem will output only one of the two candidate optimal solutions (shown in green) and the corresponding gray shaded region will form the set $\widehat{\mathcal{C}}$.}
\end{figure}

\begin{lemma} \label{lem:SolutionExist}
The set $\widehat{\mathcal{C}}'$ is nonempty.
\end{lemma}

\begin{theorem}
\label{thm:finiteSolutions}
The set of optimal solutions to \eqref{eqn:IOP} is finite.
\end{theorem}

\begin{proof}
    We start by characterising the set $\widehat{\mathcal{C}}'$. Consider the minimization problem that describes $\widehat{\mathcal{C}}'$:
    \begin{equation}
    \label{eqn:IOP_innerMin}
    \begin{aligned}
        \minimize\limits_{\hat{c} \in \mathbb{R}^n \setminus \{0\}, \, \hat{x}} \quad & \sum\limits_{i \in \mathcal{I}}\sum\limits_{j \in \mathcal{J}_i}\, \lVert x_{ij} - \hat{x}_i \rVert \\
        \st \quad & \hat{x}_i \in \argmin\limits_{\tilde{x} \in \mathbb{R}^n} \left\lbrace \hat{c}^{\top} \tilde{x}: A_i \tilde{x} \leq b_i \right\rbrace \cap \mathcal{V}_i \quad \forall \, i \in \mathcal{I}.
    \end{aligned}
    \end{equation}
    As stated in Lemma \ref{lem:SolutionExist}, the solution set of problem \eqref{eqn:IOP_innerMin} is nonempty. Furthermore, as $\mathcal{V}_i$ is finite and $\hat{x}_i \in \mathcal{V}_i$, the set of feasible $\hat{x}_i$ is finite. Since we also have a finite number of experiments, the set of feasible $\hat{x}$ to \eqref{eqn:IOP_innerMin} is finite. Hence, the set of optimal $\hat{x}$, which we denote by $\widehat{\mathcal{X}}^*$, is finite. From Theorem \ref{thm:IntersectCones}, the set of feasible $\hat{c}$ for any $\hat{x} \in \widehat{\mathcal{X}}^*$ is $\bigcap\limits_{i \in \mathcal{I}} \mathrm{cone}\left(\{-a_{it}\}_{t \in \mathcal{T}(\hat{x}_{i})}\right) \setminus \{0\}$, which is also the set of optimal $\hat{c}$ for that given optimal $\hat{x}$ as $\hat{c}$ does not appear in the objective function of \eqref{eqn:IOP_innerMin}. It follows that, considering all $\hat{x} \in \widehat{\mathcal{X}}^*$, the set $\widehat{\mathcal{C}}'$ can be expressed as follows:
    \begin{equation}
    \label{eqn:general_chat}
        \widehat{\mathcal{C}}' = \mathlarger{\bigcup}\limits_{\hat{x} \in \widehat{\mathcal{X}}^*}\left\lbrace\bigcap\limits_{i \in \mathcal{I}} \mathrm{cone}\left(\{-a_{it}\}_{t \in \mathcal{T}(\hat{x}_{i})}\right)\right\rbrace \setminus \{0\}.
    \end{equation}
    
    Problem \eqref{eqn:IOP} can then be solved by solving the following problem for every $\hat{x} \in \widehat{\mathcal{X}}^*$:
    \begin{equation}
    \label{eqn:IOP_outerMin}
    \begin{aligned}
        \minimize\limits_{\hat{c}} \quad & \lVert \bar{c} - \hat{c} \rVert_2^2 \\
        \st \quad & \hat{c} \in \bigcap\limits_{i \in \mathcal{I}} \mathrm{cone}\left(\{-a_{it}\}_{t \in \mathcal{T}(\hat{x}_{i})}\right),
    \end{aligned}
    \end{equation}
    which is the minimization of a strictly convex function over a convex feasible region and hence has a unique optimal solution, which we denote by $\tilde{c}(\hat{x})$. The set $\widetilde{\mathcal{C}} := \{\tilde{c}(\hat{x})\}_{\hat{x} \in \widehat{\mathcal{X}}^*}$ is finite. Hence, the set of optimal solutions to \eqref{eqn:IOP}, which can be expressed as
    \begin{equation}
        \widetilde{\mathcal{C}}^* := \left\lbrace \hat{c}: \hat{c} \in \widetilde{\mathcal{C}}, \; \lVert \bar{c} - \hat{c} \rVert_2^2 = \min_{\hat{x} \in \widehat{\mathcal{X}}^*} \lVert \bar{c} - \tilde{c}(\hat{x}) \rVert_2^2 \right\rbrace,
    \end{equation}
    is also finite.
\end{proof}

\begin{condition}
\label{asp:uniqueX}
There is a unique optimal $\hat{x}$ to problem \eqref{eqn:IOP_innerMin}.
\end{condition}

At first glance, Condition \ref{asp:uniqueX} seems to be very restrictive. But in fact, it holds in almost all practical instances. Consider an instance in which at any optimal solution to \eqref{eqn:IOP_innerMin}, there is a unique optimal $\hat{x}_i$ for each experiment $i$ except for one experiment $p$. Let $\hat{x}^*_i$ be the unique optimal $\hat{x}_i$ for all $i \in \mathcal{I} \setminus \{p\}$, and let $\widehat{\mathcal{X}}^*_p$ denote the set of multiple optimal $\hat{x}_p$. Then, the following two conditions have to hold: (i) The resulting loss associated with experiment $p$, i.e. $\sum_{j \in \mathcal{J}_p} \lVert x_{pj} - \hat{x}_p \rVert_2^2$, is the same for all $\hat{x}_p \in \widehat{\mathcal{X}}^*_p$. (ii) For every $\hat{x}_p \in \widehat{\mathcal{X}}^*_p$, there exists a $\hat{c}$ that renders all $\hat{x}^*_i$, $i \in \mathcal{I} \setminus \{p\}$, and $\hat{x}_p$ optimal for the FOP, which is equivalent to the following condition:
\begin{equation}
\label{eqn:nomultiplesoln}
    \mathrm{cone}\left(\{-a_{pt}\}_{t \in \mathcal{T}(\hat{x}_{p})}\right) \mathlarger{\bigcap} \left(\bigcap\limits_{i \in \mathcal{I} \setminus \{p\}} \mathrm{cone}\left(\{-a_{it}\}_{t \in \mathcal{T}(\hat{x}^*_{i})}\right)\right) \neq \{0\} \quad \forall \, \hat{x}_p \in \widehat{\mathcal{X}}^*_p.
\end{equation}
While already the first condition is very unlikely to hold in practice, the second also becomes more improbable for $|\widehat{\mathcal{X}}^*_p| > 1$ as the number of experiments increases. Hence, we conclude that the case of $|\widehat{\mathcal{X}}^*| > 1$ is highly unlikely, which is consistent with our observation in our computational experiments. In theory, however, especially if $|\mathcal{I}|$ is small, there is the possibility that Condition \ref{asp:uniqueX} does not hold. In the next subsection, we present a two-phase algorithm that is guaranteed to find the complete set of optimal solutions to \eqref{eqn:IOP} even if Condition \ref{asp:uniqueX} does not hold.

\begin{corollary}
\label{cor:uniqueSolution}
If Condition \ref{asp:uniqueX} holds, \eqref{eqn:IOP} has a unique solution.
\end{corollary}


Thus, in most cases, solving $\eqref{eqn:IOP}$ will result in a unique $\hat{c}$ that is closest to the prior belief $\bar{c}$. This is in contrast to other approaches in the literature where a variant of \eqref{eqn:IOP_innerMin} is solved that results in a set of solutions from which a $\hat{c}$ is randomly selected. Note that while all $\hat{c} \in \widehat{\mathcal{C}}$ achieve the same prediction accuracy on the training set, they may not show the same performance on unseen data. Our approach of incorporating a prior belief $\bar{c}$ resolves this ambiguity and determines whether $\bar{c}$ is in the admissible set or, if not, how much it is outside the admissible set. This is a generally desirable feature in practice.

\subsection{Two-Phase Algorithm}

We start by presenting tractable reformulations of \eqref{eqn:IOP_innerMin} and \eqref{eqn:IOP_outerMin}, and then show how they can be combined to obtain the set of optimal solutions to \eqref{eqn:IOP}.

Applying a KKT-based approach, we obtain the following reformulation of \eqref{eqn:IOP_innerMin}:
\begin{varsubequations}{P1}
\label{eqn:P1}
\begin{align}
    \minimize\limits_{\hat{c}, \, \hat{x}, \, s, \, \lambda, \, z, \, w, \, \hat{c}^+, \hat{c}^-} \quad & \sum\limits_{i \in \mathcal{I}}\sum\limits_{j \in \mathcal{J}_{i}} \lVert x_{ij} - \hat{x}_{i} \rVert \\
    \st \quad \quad & \hat{c} + A_i^{\top} \lambda_{i} = 0 \quad \forall \, i \in \mathcal{I} \label{eqn:KKTstart} \\
    & A_i\hat{x}_{i} + s_{i} = b_i \quad \forall \, i \in \mathcal{I}  \\ 
    & \lambda_{i} \leq M z_{i} \quad \forall \, i \in \mathcal{I} \\
    & s_{i} \leq M (e-z_{i}) \quad \forall \, i \in \mathcal{I}  \label{eqn:KKTend} \\
    & e^{\top} z_i \geq n \quad \forall \, i \in \mathcal{I} \label{eqn:Vertex_const} \\
    & \hat{c} = \hat{c}^+ - \hat{c}^- \label{eqn:pnorm_const1} \\
    & \hat{c}^+ \leq w  \\
    & \hat{c}^- \leq e-w \\
    & e^{\top} (\hat{c}^+ + \hat{c}^-) = 1  \label{eqn:pnorm_const2} \\
    & \hat{x}_i \in \mathbb{R}^n, \, s_i \in \mathbb{R}_+^m, \, \lambda_i \in \mathbb{R}_+^m, \, z_i \in \{0,1\}^m \quad \forall \, i \in \mathcal{I} \\
    & \hat{c} \in \mathbb{R}^n, \hat{c}^+ \in \mathbb{R}^n_+ , \hat{c}^- \in \mathbb{R}^n_+, \, w \in \{0,1\}^n,
\end{align}
\end{varsubequations}
where $M$ is a sufficiently large parameter. Constraints \eqref{eqn:KKTstart}--\eqref{eqn:KKTend} correspond to the KKT optimality conditions of the lower-level problems in \eqref{eqn:IOP_innerMin}, \eqref{eqn:Vertex_const} ensures that $\hat{x}_i$ is a vertex of the polyhedron $\{\tilde{x}: A_i \tilde{x} \leq b_i\}$, and \eqref{eqn:pnorm_const1}--\eqref{eqn:pnorm_const2} represent a linearization of the condition $\lVert \hat{c} \rVert_1 = 1$, which excludes $\hat{c} = 0$ from the set of feasible solutions. Note that \eqref{eqn:P1} is not an exact reformulation of \eqref{eqn:IOP_innerMin} since the constraint $\lVert \hat{c} \rVert_1 = 1$ cuts off more than just the single point $\hat{c} = 0$. However, the following property suggests that this restriction does not affect the solution set of \eqref{eqn:IOP}.

\begin{lemma}
\label{lem:epsilon}
Problems \eqref{eqn:IOP_innerMin} and \eqref{eqn:P1} have the same set of feasible $\hat{x}$.
\end{lemma}


An exact reformulation of \eqref{eqn:IOP_outerMin} directly follows from the definition of polyhedral cones:
\begin{equation}
\label{eqn:P2}
\tag{P2}
\begin{aligned}
    \minimize\limits_{\hat{c}, \, \gamma} \quad & \lVert \bar{c} - \hat{c} \rVert_2^2 \\
    \st \quad & \hat{c}  = -\sum\limits_{t \in \mathcal{T}(\hat{x}_i)} \gamma_{it} \,  a_{it} \quad \forall \, i \in \mathcal{I} \\
    & \gamma_{it} \geq 0 \quad \forall \, i \in \mathcal{I}, \, t \in \mathcal{T}(\hat{x}_i),
\end{aligned}
\end{equation}
where the constraints require that $\hat{c}$ can be expressed as a conic combination of $-a_{it}$, $t \in \mathcal{T}(\hat{x}_i)$, for every $i \in \mathcal{I}$.

\begin{proposition}
\label{prp:TwoPhase}
If Condition \ref{asp:uniqueX} holds, the optimal solution to \eqref{eqn:IOP} can be obtained by solving \eqref{eqn:P1}, which provides the unique optimal $\hat{x}^*$, and subsequently solving \eqref{eqn:P2} with $\hat{x} = \hat{x}^*$.
\end{proposition}


In the general case where Condition \ref{asp:uniqueX} does not necessarily hold (or where we do not know in advance that it holds), the complete set $\widehat{\mathcal{X}}^*$ can be recovered by re-solving \eqref{eqn:P1} with added integer cuts of the following form until the objective function value changes:
\begin{equation}
\label{eqn:IntegerCut}
    \sum_{i \in \mathcal{I}} \sum_{t \in \mathcal{T}(\hat{x}_i)} z_{it} \leq n|\mathcal{I}| - 1.
\end{equation}
Since the set of active constraints at a particular $\hat{x}_i$ is given by the values of the binary variables $z_i$, we use \eqref{eqn:IntegerCut} to impose the condition that for a different optimal solution to exist, at least one of the $\hat{x}_i$ has to induce a different set of active constraints. A two-phase algorithm incorporating integer cuts to obtain the complete set of optimal solutions to \eqref{eqn:IOP} as described in the proof of Theorem \ref{thm:finiteSolutions} is shown in Algorithm \ref{alg:Two-phase}.

\begin{algorithm}
\begin{algorithmic}[1]
    \Statex \textsc{Phase 1:}
    \State solve \eqref{eqn:P1}, obtain $\hat{x}^*$, $r^* \gets \sum_{i \in \mathcal{I}} \sum_{j \in \mathcal{J}_{i}} \lVert x_{ij} - \hat{x}^*_{i} \rVert$
    \State initialize: $k \gets 1$, $r^1 \gets r^*$, $\hat{x}^1 \gets \hat{x}^*$
    \While{$r^k = r^*$}
        \State add integer cut \eqref{eqn:IntegerCut} for $\hat{x}^k$ to \eqref{eqn:P1}
        \State solve \eqref{eqn:P1}, obtain $\hat{x}^*$, $\hat{x}^{k+1} \gets \hat{x}^*$, $r^{k+1} \gets \sum_{i \in \mathcal{I}} \sum_{j \in \mathcal{J}_{i}} \lVert x_{ij} - \hat{x}^{k+1}_{i} \rVert$
        \State $k \gets k+1$
    \EndWhile
    \Statex \textsc{Phase 2:}
    \State $K \gets k-1$
    \ForAll{$k = 1, \dots, K$}
        \State solve \eqref{eqn:P2} with $\hat{x} = \hat{x}^k$, obtain $\hat{c}^*$,  $\hat{c}^k \gets \hat{c}^*$, $v^k \gets \lVert \bar{c} - \hat{c}^k \rVert_2^2$
    \EndFor
    \State \Return all $\hat{c}^k$ for which $v^k = \min_{k' = 1,\dots,K} v^{k'}$
    \caption{Two-phase algorithm for solving \eqref{eqn:IOP}.}
    \label{alg:Two-phase}
\end{algorithmic}
\end{algorithm}

\pgfplotsset{compat=1.15}
\usetikzlibrary{arrows}

\begin{remark} \label{rem:high_noise}
Problems \eqref{eqn:P1} and \eqref{eqn:P2} do not include any additional constraints on the values of $c$ (and therefore $\hat{c}$) apart from it being a nonzero vector in \eqref{eqn:P1}. In practice, it is likely that we have additional information about the true cost vector. For example, if the cost coefficients to be estimated represent costs of items, one can safely assume them to be strictly positive. Such information can be directly incorporated into \eqref{eqn:P1} and \eqref{eqn:P2} in the form of additional constraints, which can further restrict the admissible set and help obtain a reasonable cost estimate. This can be especially useful in situations where the observed data has a high variance or the vertices of the polyhedron are very close to each other. We illustrate this point through an example in {the online supplement, section C}. 
\end{remark}

\begin{remark} \label{rem:diff_norm}
So far, we have defined the loss function, i.e. the objective function of \eqref{eqn:P1}, to be the sum of a general norm of the residuals. Typically, a $p$-norm is used. While the 2-norm seems to work well in most cases, ideally, the norm should be chosen based on the type of noise in the data. Specifically, if the observations are known to be prone to outliers, the 1-norm can be used as it is known to be robust against outliers; whereas, if the data is highly accurate and the hypothesis of a linear objective function is being tested, one might want to use the $\infty$-norm to minimize the worst-case residuals.
\end{remark}

\section{A Sequential Decomposition Algorithm for \eqref{eqn:P1}} \label{sec:DecomAlg}

The phase-1 problem \eqref{eqn:P1} is an MILP or MINLP (depending on the choice of loss function) whose size increases with the number of experiments, inducing computational challenges when the dataset is large. Therefore, in the following, under the assumption that Condition \ref{asp:uniqueX} holds everywhere, we present an exact decomposition algorithm that can substantially reduce the computation time in large instances.

Notice that \eqref{eqn:P1} has a clear decomposable structure. Specifically, $\hat{c}$ act as linking variables such that with fixed $\hat{c}$, \eqref{eqn:P1} decomposes into $|\mathcal{I}|$ independent subproblems, one for each experiment. However, due to the nonconvex nature of the subproblems, traditional Benders-type decomposition methods cannot be directly applied to solve the problem to provable optimality. Instead, we develop a decomposition method that is more akin to Lagrangean decomposition but exploits the structure of the problem such that the exact optimal solution can be obtained. We start by presenting some properties of \eqref{eqn:P1} that form the basis of our solution algorithm.

\begin{lemma}
\label{lem:valLB}
Given a set of experiments $\{1, \ldots, N\}$, let $\eqref{eqn:P1}_i$ with $i \leq N$ be an instance of \eqref{eqn:P1} with $\mathcal{I} = \{ i \}$ and let its optimal value be $\bar{r}_i^*$. If $r^*$ denotes the optimal value of \eqref{eqn:P1} with $\mathcal{I} = \{ 1, \dots, N\}$, then $\sum\limits_{i=1}^N \bar{r}_i^* \leq r^*$.
\end{lemma}

\begin{corollary} \label{cor: SequentialUpdate}
Let $\eqref{eqn:P1}_{[i]}$ be an instance of \eqref{eqn:P1} with $\mathcal{I} = \{1, \ldots, i\}$ and $\bar{r}^*_{[i]}$ be its optimal value. Then, $\bar{r}^*_{[i-1]} + \bar{r}^*_{i} \leq \bar{r}^*_{[i]}$.
\end{corollary}


Our sequential decomposition approach for \eqref{eqn:P1} is a direct consequence of Corollary \ref{cor: SequentialUpdate}. We first provide an intuitive description of the method using the illustrative example shown in Figure \ref{fig:DecompAlg}, which involves three experiments. The main idea is to solve \eqref{eqn:P1} \textit{sequentially}, i.e. one experiment by one experiment, instead of directly solving the full-space problem considering all experiments. With only one experiment, \eqref{eqn:P1} reduces to the problem of projecting the noisy observations onto the vertex that minimizes the loss. When the level of noise is small, it is likely that this already results in the solution that is optimal for the full problem, as it is the case for experiments 1 and 2 in Figure \ref{fig:DecompAlg}. In experiment 3, however, the level of noise is so high that the projection yields the wrong vertex. A solution to the full problem involving all experiments has to provide a $\hat{c}$ that renders all $\hat{x}_i$ optimal, which implies that the resulting admissible set $\widehat{\mathcal{C}}$ must not only contain 0. As shown at the top right of Figure \ref{fig:DecompAlg}, while the true cost vector $c$ lies in the intersection of the two cones resulting from the correct projections, the intersection of all three cones including the one associated with the incorrect projection is $\{0\}$. When $\widehat{\mathcal{C}} = \{0\}$, we solve \eqref{eqn:P1} considering all experiments up to that point jointly in order to correct the projections such that the resulting $\widehat{\mathcal{C}}$ is a proper polyhedral cone (see bottom left of Figure \ref{fig:DecompAlg}). 

\begin{figure}[ht]
    \includegraphics[clip, trim=0 5.5cm 0 5cm, width=\textwidth]{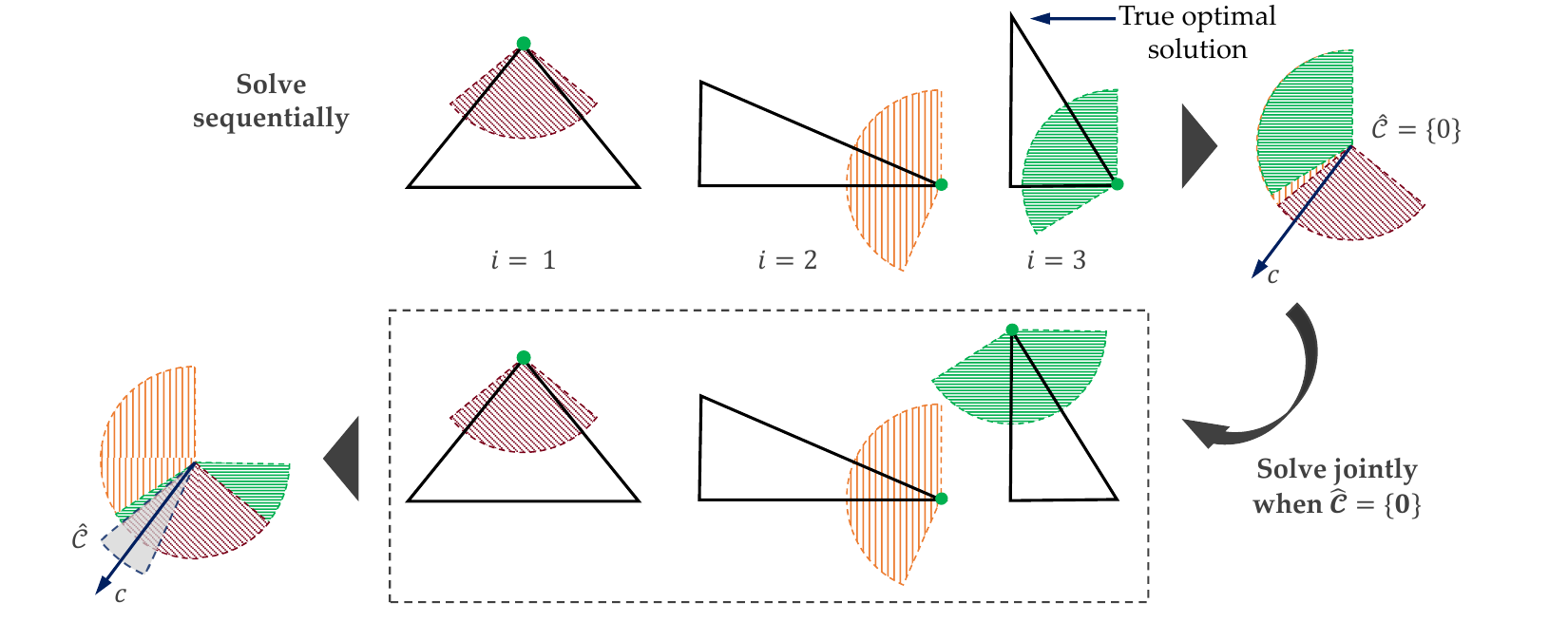}
    \centering
    \caption{{(Color online)} Depiction of the proposed decomposition algorithm. Note that only the vertex projections are shown (filled green circles), the corresponding noisy observations are omitted for the sake of clearer visualization.}
    \label{fig:DecompAlg}
\end{figure}

Detecting such infeasible projections is crucial for the proposed decomposition algorithm, and it turns out that this can be accomplished by solving a very efficient feasibility problem, which we establish in the following proposition. 

\begin{proposition} \label{prp:DecompValidity}
Let $(\hat{x}_1, \ldots, \hat{x}_{\ell-1})$ and $\hat{x}_{\ell}$ be feasible solutions to problems $\eqref{eqn:P1}_{[\ell-1]}$ and $\eqref{eqn:P1}_{\ell}$, respectively. Then, $(\hat{x}_1, \ldots, \hat{x}_{\ell})$ is a feasible solution to $\eqref{eqn:P1}_{[\ell]}$ if the following problem is feasible:
\begin{equation}
    \label{eqn:FP}
    \tag*{$(\mathrm{FP})_{\ell}$}
    \begin{aligned}
        \minimize\limits_{\hat{c}, \hat{c}^+, \hat{c}^-, \gamma, w} \quad & 0 \\
        \st \quad & \hat{c}  = \sum\limits_{t \in \mathcal{T}(\hat{x}_i)} \gamma_{it}(-a_{it}) \quad \forall \, i \in \{1, \ldots, \ell \} \\
        & {\hat{c}} = {\hat{c}}^+ - {\hat{c}}^-  \\
        & {\hat{c}}^+ \leq w  \\
        & {\hat{c}}^- \leq e-w \\
        & e^{\top} ({\hat{c}}^+ + {\hat{c}}^-) = 1   \\        
        & \gamma_{it} \geq 0 \quad \forall \, i \in \{1, \ldots, \ell \}, t \in \mathcal{T}(\hat{x}_i) \\
        & \hat{c} \in \mathbb{R}^n,\,{\hat{c}}^+ \in \mathbb{R}^n_+,\,{\hat{c}}^- \in \mathbb{R}^n_+ ,\, w \in \{0,1\}^n.
    \end{aligned}
\end{equation}
\end{proposition}

Note that \ref{eqn:FP} is a feasibility problem (as opposed to an optimization problem), and therefore, the objective function is arbitrarily set to a constant 0. 

\begin{corollary}
Let $(\hat{x}_1, \ldots, \hat{x}_{\ell-1})$ and $\hat{x}_{\ell}$ be optimal solutions to problems $\eqref{eqn:P1}_{[\ell-1]}$ and $\eqref{eqn:P1}_{\ell}$, respectively. Then, $(\hat{x}_1, \ldots, \hat{x}_{\ell})$ is an optimal solution to $\eqref{eqn:P1}_{[\ell]}$ if \ref{eqn:FP} is feasible.
\end{corollary}


Although \ref{eqn:FP} is an MILP, its number of binary variables does not change with the number of data points, which allows it to remain tractable for instances with many experiments. Therefore, in the algorithm, after every experiment $\ell$, we solve \ref{eqn:FP} to check the validity of the partial solution obtained up to that point for problem $\eqref{eqn:P1}_{[\ell]}$. If at some point, \ref{eqn:FP} is found infeasible, we solve the full problem $\eqref{eqn:P1}_{[\ell]}$. Here, it is important to note that since the solution obtained for $\eqref{eqn:P1}_{[\ell - 1]}$ in the previous iteration is guaranteed to be feasible for $\eqref{eqn:P1}_{[\ell]}$, it can be used to warm-start the solver, which is another crucial factor with regard to the computational efficiency of the algorithm. In practice, with low level of noise or a sufficiently large number of samples $|\mathcal{J}_i|$ for each experiment $i \in \mathcal{I}$, one can expect the single-experiment projection to be accurate most of the time. As a result, one may only need to solve a small number of problems involving multiple experiments. The pseudocode for the complete algorithm for solving \eqref{eqn:IOP} with this sequential decomposition approach is shown in Algorithm \ref{alg:Online Algorithm}. 

\begin{algorithm}
\begin{algorithmic}[1]
    \State initialize: $\widehat{\mathcal{C}} \gets \mathbb{R}^n$
    \ForAll{$ \ell \in \{1, \dots, N\}$}
        \State solve $\eqref{eqn:P1}_{\ell}$, obtain $\hat{x}_{\ell}^*$
        \State $\widehat{\mathcal{C}} \gets \widehat{\mathcal{C}} \, \bigcap \, \mathrm{cone}\left(\{-a_{\ell t}\}_{t \in \mathcal{T}(\hat{x}^*_{\ell})}\right)$
        \State solve \ref{eqn:FP}
        \If {\ref{eqn:FP} is infeasible}
            \State  Solve (P1)$_{[\ell]}$, warm-start with $(\hat{x}_1^*, \ldots, \hat{x}_{\ell - 1}^*)$, obtain $\hat{x}_{[\ell]}^* = (\hat{x}_1^*, \ldots, \hat{x}_{\ell}^*)$
            \State  $\widehat{\mathcal{C}} \gets \bigcap\limits_{i \in \{1, \ldots, \ell \}} \mathrm{cone}\left(\{-a_{it}\}_{t \in \mathcal{T}(\hat{x}^*_{i})}\right)$
        \EndIf
    \EndFor
        \State solve \eqref{eqn:P2}, obtain $\hat{c}^*$
    \caption{Decomposition-based algorithm for solving \eqref{eqn:IOP}}
    \label{alg:Online Algorithm}
\end{algorithmic}
\end{algorithm}

\paragraph{Online Inverse Optimization}
So far, we have assumed that all data are available prior to the inverse optimization process. However, in many situations, it is reasonable to expect the data from different experiments to become available at different time points. This has motivated the development of online learning frameworks for inverse optimization where the cost estimate is updated as new data arrive \citep{Barmann2017, Dong2018, Shahmoradi2019}. A naive extension of our two-phase framework for online learning can be to repeatedly solve \eqref{eqn:IOP} with the new data added to it. However, note that Algorithm \ref{alg:Online Algorithm} provides a much more efficient mechanism to address this problem, and can be adapted as a method for inverse optimization in an online setting. We include such an adaptation in {the online supplement, section B}.

\section{Online Adaptive Sampling} \label{sec:ASP}

In this section, we introduce a novel adaptive sampling approach that can be used to reduce the number of experiments required to obtain a good cost estimate in an online setting. Our method derives from the geometrical understanding of \eqref{eqn:IOP} discussed in the previous sections. As a motivating example, let us revisit the case with three experiments shown in Figure \ref{fig:DecompAlg}. Notice that the admissible set $\widehat{\mathcal{C}}$ obtained with the first two experiments is the same as the one obtained with all three experiments. The reason is that the cone associated with the projected solution from experiment 3 is a superset of the cone from experiment 1; hence, the intersection of cones does not change with the third experiment. This means that experiment 3 cannot help improve the cost estimate or, more precisely, our confidence in the cost estimate, which only increases if the size of $\widehat{\mathcal{C}}$ decreases. Therefore, the goal here is to use the current best cost estimate to design subsequent experiments in a way such that the size of $\widehat{\mathcal{C}}$ will likely be further reduced. In the following, we formulate the adaptive sampling problem as an optimization problem, and devise an efficient heuristic method to solve it.

\subsection{Mathematical Formulation}
Consider the point in an online inverse optimization process at which $\ell - 1$ experiments are completed and the problem is to choose input parameters $(A_{\ell}, b_{\ell}) \in \Pi$ for the $\ell$th experiment. Let $\mathcal{I}$ be the set of the first $\ell-1$ experiments, and let $\widehat{\mathcal{C}}$ be the admissible set obtained using these experiments. Suppose $\widehat{\mathcal{C}}_\ell$ denotes the (unknown) cone formed by the active constraints of the FOP with its feasible region represented by the polyhedron $\{\tilde{x}: A_{\ell} \tilde{x} \leq b_{\ell}\}$. The goal as illustrated in the above example is to choose $A_\ell $ and $b_\ell$ such that $\widehat{\mathcal{C}} \cap \widehat{\mathcal{C}}_\ell \subset \widehat{\mathcal{C}}$. While it is impossible to determine $\widehat{\mathcal{C}}_\ell $ without knowing $c$ exactly, we make use of a randomly sampled vector $\tilde{c}_{\ell-1}$ from $\widehat{\mathcal{C}}$ to predict which constraints may become active for given sets of potential input parameters. 

Assuming we have an accurate estimate of $\widehat{\mathcal{C}}_{\ell}$, we utilize the fact that $\widehat{\mathcal{C}} \cap \widehat{\mathcal{C}}_{\ell} \subset \widehat{\mathcal{C}}$ if and only if we can find a point that belongs to $\widehat{\mathcal{C}}$ but not to $\widehat{\mathcal{C}}_{\ell}$. As we illustrate in Figure \ref{fig:ConeInt}, the existence of such a point can be determined by obtaining the minimum-distance projections of all points in $\widehat{\mathcal{C}}$ on $\widehat{\mathcal{C}}_{\ell}$ and searching for one that yields a nonzero projection distance. In the figure, $\bar{y}$ denotes the minimum-distance projection of a point $y$ from $\widehat{\mathcal{C}}$ on $\widehat{\mathcal{C}}_{\ell}$. In case of Figure \ref{fig:coneoverlap}, all points $y \in \widehat{\mathcal{C}}$ are such that $\bar{y} = y$ and hence, the intersection of the two cones does not result in a reduction in the size of the admissible set. In Figure \ref{fig:noconeoverlap}, however, one can find a $y$ and its corresponding $\bar{y}$ with $\lVert y - \bar{y} \rVert > 0$ and therefore, the intersection of the two cones is a strict subset of $\widehat{\mathcal{C}}$. 

\begin{figure}[h!]
\centering
\subfloat[$\widehat{\mathcal{C}} = \widehat{\mathcal{C}} \cap \widehat{\mathcal{C}}_{\ell}$]{
    \label{fig:coneoverlap}
    \pgfplotsset{compat=1.15}
    \usetikzlibrary{arrows}
    \begin{tikzpicture}[line cap=round,line join=round,>=triangle 45,x=0.1cm,y=0.1cm]
        \begin{axis}[
            x=4cm,y=4cm,
            axis lines=none,
            ymajorgrids=false,
            xmajorgrids=false,
            xmin=-.10,
            xmax=1.5,
            ymin=-0.1,
            ymax=1.15,
            ]
            \clip(-0.1,-0.1) rectangle (4,4);
            \draw [->,line width=2pt,color=red, dashed] (0,0) -- (0,1);
            \draw [->,line width=2pt,color=red, dashed] (0,0) -- (1,0);
            \draw [->,line width=2pt,color=blue] (0,0) -- (0.5547,0.83205);
            \draw [->,line width=2pt,color=blue] (0,0) -- (0.995037,0.0995037);
            \draw [fill=black] (0.5,0.6) circle (1pt) node[right] {$y = \bar{y}$};
            \filldraw[blue, draw=none, opacity = 0.2] (0,0) -- (0.995037,0.0995037) arc (6.2:56.7:1) -- (0,0);
            \filldraw[red, draw=none, opacity = 0.2] (0,0) -- (1,0) arc (0:90:1) -- (0,0);
            \node[coordinate, pin={[align=left,pin distance = 4mm]60:{$\widehat{\mathcal{C}}_{\ell}$}}]
                           at (0.4,0.87) {};
            \node[coordinate, pin={[align=left,pin distance = 4mm]30:{$\widehat{\mathcal{C}} = \widehat{\mathcal{C}} \cap \widehat{\mathcal{C}}_{\ell}$}}]
                           at (0.8,0.4) {};
        \end{axis}
    \end{tikzpicture}}
\subfloat[$\widehat{\mathcal{C}} \cap \widehat{\mathcal{C}}_{\ell} \subset \widehat{\mathcal{C}}$]{
    \label{fig:noconeoverlap}
    \begin{tikzpicture}[line cap=round,line join=round,>=triangle 45,x=0.1cm,y=0.1cm]
        \begin{axis}[
            x=4cm,y=4cm,
            axis lines=none,
            ymajorgrids=false,
            xmajorgrids=false,
            xmin=-.10,
            xmax=1.5,
            ymin=-0.1,
            ymax=1.15,
            ]
            \clip(-0.1,-0.1) rectangle (4,4);
            \draw [->,line width=2pt,color=red, dashed] (0,0) -- (0.8944,0.44);
            \draw [->,line width=2pt,color=red, dashed] (0,0) -- (1,0);
            \draw [->,line width=2pt,color=blue] (0,0) -- (0.5547,0.83205);
            \draw [->,line width=2pt,color=blue] (0,0) -- (0.99504,0.0995);
            \draw [fill=black] (0.5,0.6) circle (1pt) node[right] {$y$};
            \draw [fill=black] (0.6402,0.31496) circle (1pt) node[above] {$\bar{y}$};
            \draw [line width = 0.5, dotted, color = black] (0.5,0.6) --(0.6402,0.31496);
            \filldraw[red, draw=none, opacity = 0.2] (0,0) -- (1,0) arc (0:26.1:1) -- (0,0) ;
            \filldraw[blue, draw=none, opacity = 0.2] (0,0) -- (0.995037,0.0995037) arc (6.2:56.7:1) -- (0,0);
            \node[coordinate, pin={[align=left,pin distance = 4mm]30:{$\widehat{\mathcal{C}}$}}]
                           at (0.7,0.6) {};
            \node[coordinate, pin={[align=left,pin distance = 4mm]0:{$\widehat{\mathcal{C}}_{\ell}$}}]
                           at (0.98,0.05) {};
            \node[coordinate, pin={[align=left,pin distance = 4mm]20:{$\widehat{\mathcal{C}} \cap \widehat{\mathcal{C}}_{\ell}$}}]
                           at (0.88,0.25) {};
        \end{axis}
    \end{tikzpicture}}
\caption{{(Color online)} The cone formed by the blue vectors represents $\widehat{\mathcal{C}}$ whereas the red dashed vectors form the cone $\widehat{\mathcal{C}}_{\ell}$. The dotted line in the second case shows the minimum 2-norm projection of $y$ on $\widehat{\mathcal{C}}_{\ell}$.}
\label{fig:ConeInt}
\end{figure}
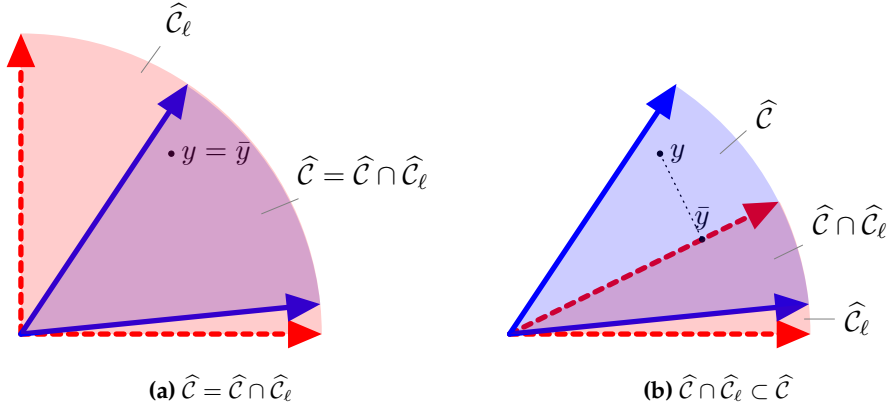


Given $\widehat{\mathcal{C}}$ and $\widehat{\mathcal{C}}_{\ell}$, we find the $y \in \widehat{\mathcal{C}}$ with the maximum minimum-projection distance to $\widehat{\mathcal{C}}_{\ell}$ and use this distance as a measure for the difference between $\widehat{\mathcal{C}}$ and $\widehat{\mathcal{C}} \cap \widehat{\mathcal{C}}_{\ell}$. Hence, as we wish to reduce the size of the admissible set as much as possible, we need to choose $A_{\ell}$ and $b_{\ell}$ for which this distance is the largest among all $(A_{\ell}, b_{\ell}) \in \Pi$. We formulate the problem of finding such $A_{\ell}$ and $b_{\ell}$ as the following optimization problem:
\begin{varsubequations}{ASP}
\label{eqn:ASP}
\begin{align}
    \maximize\limits_{y, \gamma, z, A_{\ell}, b_{\ell}, \hat{x}_{\ell}} \quad & \min\limits_{\eta, \bar{y}, \bar{\gamma}} \; \left\{ \eta : \eta = \lVert \; y - \bar{y} \; \rVert_1, \,  \bar{y} = \sum\limits_{k \in \mathcal{K}} \bar{\gamma}_k z_k a_{\ell k}, \, \bar{\gamma}_{k} \geq 0 \; \forall \,  k \in \mathcal{K} \right \} \label{eqn:ASPobj} \\
    \st \quad & y = \sum_{t \in \mathcal{T}(\hat{x}_i)} \gamma_{it}\left(-\frac{a_{it}}{\lVert a_{it} \rVert}_2\right) \quad \forall \, i \in \mathcal{I}  \label{eqn:ASPConst1} \\
    &  0 \leq  \gamma_{it} \leq \epsilon \quad \forall \, i \in \mathcal{I}, t \in \mathcal{T}(\hat{x}_i) \label{eqn:ASPConst2} \\
    & (A_{\ell}, b_{\ell}) \in \Pi  \label{eqn:ASPConst3} \\
    & z_k = \Bigg\{\begin{array}{lr}
        1,\;\; \text{if } k \in \mathcal{T}(\hat{x}_{\ell}) \; \text{where } \hat{x}_{\ell} \in \argmin\limits_{\tilde{x} \in \mathbb{R}^n} \left\lbrace \tilde{c}_{\ell-1}^{\top} \tilde{x}: A_{\ell} \tilde{x} \leq b_{\ell} \right\rbrace \\
        0, \;\; \text{otherwise }  \\
        \end{array} \quad \forall \, k \in \mathcal{K}, \label{eqn:ASPConst5}
\end{align}
\end{varsubequations}
where $\mathcal{K}$ is the set of constraints for $A_{\ell}x \leq b_{\ell}$. Constraints \eqref{eqn:ASPConst1} state that $y$ is a point in $\widehat{\mathcal{C}}$. Note that the vectors forming cone $\widehat{\mathcal{C}}_{i}$ for each $i \in \mathcal{I}$ have been normalized to ensure that the lengths of the vectors do not bias the value of $\eta$. The upper bound $\epsilon$ in \eqref{eqn:ASPConst2} ensures that the problem remains bounded. The value of $\epsilon$ can be chosen arbitrarily and does not affect the choice of optimal $A_\ell $ and $b_\ell $ for the problem. Constraint \eqref{eqn:ASPConst3} imposes the restriction that input parameters $A_\ell $ and $b_\ell $ must be chosen from the set $\Pi$. As stated in constraints \eqref{eqn:ASPConst5}, $z_k$ is a binary variable that equals $1$ if $\tilde{c}_{\ell-1}$ predicts that the $k{\mathrm{th}}$ constraint will be active. Finally, the objective of $\eqref{eqn:ASP}$ is to maximize the minimum $1$-norm distance between $y$ and $\bar{y}$, with $\bar{y}$ being constrained to lie in the cone $\widehat{\mathcal{C}}_{\ell}$. We denote the distance between $y$ and $\bar{y}$ by $\eta$. 

\subsection{A Heuristic Solution Algorithm for \eqref{eqn:ASP}} 

Problem $\eqref{eqn:ASP}$ is a bilevel optimization problem with lower-level problems embedded in the objective function \eqref{eqn:ASPobj} and constraints \eqref{eqn:ASPConst5}. Since these embedded problems are LPs, \eqref{eqn:ASP} can be reformulated into a single-level problem, similar to the bilevel problems considered in previous sections. We present two such single-level reformulations in {the online supplement, section D}. Both of them take the forms of nonconvex MINLPs that are only tractable for relatively small problems. Next, we outline a heuristic approach to solve this problem. 

Instead of searching through the entire set $\Pi$ for the best set of input parameters, the proposed algorithm, Algorithm \ref{alg:RandomAdaIO}, simplifies the problem by limiting the search space to $S$ randomly sampled sets of input parameters. It then solves $S$ instances of \eqref{eqn:ASP}, each corresponding to one of the parameter sets, to choose the one that results in the largest $\eta$. Here, since the number of possible candidates for $A_{\ell}$ and $b_{\ell}$ is finite with their values being explicitly known, binary vector $z$ for each of them can be calculated before \eqref{eqn:ASP} is solved. This eliminates the lower-level problems in \eqref{eqn:ASPConst5}. Furthermore, in this algorithm, since the $S$ instances of \eqref{eqn:ASP} are independent of each other, they can be solved in parallel. This enables the use of a large $S$ to find a heuristic solution as close to the global optimum of \eqref{eqn:ASP} as possible while still keeping the problem tractable. 

\begin{algorithm}
\begin{algorithmic}[1]
    \For{$i \in \{1, \ldots, S\}$}
        \State sample $(A_i, b_i) \in \Pi $
        \State compute $z_i$ such that $z_{ik} = \Bigg\{\begin{array}{lr}
        1, & \text{if } k \in \mathcal{T}(\hat{x}_i) \; \text{where } \hat{x}_i \in \argmin\limits_{\tilde{x} \in \mathbb{R}^n} \left\lbrace \tilde{c}_{\ell-1}^{\top} \tilde{x}: A_i \tilde{x} \leq b_i \right\rbrace \\
        0, & \text{otherwise }  \\
        \end{array} $
        \State solve \eqref{eqn:ASP} with $A_{\ell} = A_i$, $b_{\ell} = b_i$, and $z = z_i$, obtain $\eta_i$
    \EndFor
    \State \Return $A_i$, $b_i$ for which $\eta_i = \text{max}_{i' = \{1, \ldots, S\}} \eta_{i'}$ 
    \caption{A heuristic algorithm for solving \eqref{eqn:ASP}}
    \label{alg:RandomAdaIO}
\end{algorithmic}
\end{algorithm}

\begin{remark}
Problem \eqref{eqn:ASP} is valid regardless of the restrictions on the input parameters $A$ and $b$, i.e. the set $\Pi$. However, there is a difference in the way the size of $\widehat{\mathcal{C}}$ is reduced when only right-hand-side (RHS) parameters $b$ are allowed to be varied compared to when the constraint matrix $A$ can also be varied. We highlight this difference in detail in {the online supplement, section E}.
\end{remark}

\section{Computational Case Studies} \label{sec:CompSt}
In this section, we apply the proposed data-driven inverse linear optimization framework to two case studies, one addressing customer preference learning and the second related to cost estimation for multiperiod production planning. Using synthetic data, we compare the computational performances of Algorithms \ref{alg:Two-phase} and \ref{alg:Online Algorithm} and evaluate the impact of adaptive sampling. All model instances were implemented in Julia v1.3.0 \citep{Bezanson2017} using the modeling language JuMP v0.18.6 \citep{Dunning2017}. All instances of \eqref{eqn:P1} were solved with a 1-norm-based objective function, and Gurobi v9.0.2 \citep{Gurobi2014} was applied to solve the resulting MILPs. Nonconvex MINLPs were solved using BARON v19.12.7 \citep{Sahinidis1996}. The Julia code for all computational experiments presented in this section is available from the IJOC GitHub software repository at \url{https://github.com/INFORMSJoC/2020.0231} \citep{Code_Github}.

\subsection{Customer Preference Learning} \label{sec:CustPref}
We consider the problem of learning customers' preferences given their purchasing decisions. The FOP here is based on the premise that with a limited budget, customers will buy products that maximize their utility. Given the price $w_p$ of each product $p$ and a budget $b$, the customer is assumed to solve the following LP:
\begin{equation}
\label{eqn:CustPref}
    \begin{aligned}
        \maximize\limits_{x \in \mathbb{R}^n_+} \quad & \sum\limits_{p \in \mathcal{P}}u_px_p \\
        \st \quad & \sum\limits_{p \in \mathcal{P}}w_px_p \leq b \\
        & x_p \leq 1 \quad \forall \, p \in \mathcal{P},
    \end{aligned}
\end{equation}
where $\mathcal{P} = \{1, \ldots, n\}$ denotes the set of $n$ products available on the market. The goal is to estimate the unknown utility function coefficients $u$ by observing the changes in the customer's decisions $x$ in response to price fluctuations. Different variations of this IOP have previously been considered by \cite{Barmann2017} and \cite{Dong2018}. \cite{Barmann2017} consider learning the utility function with deterministic data. \cite{Dong2018} account for noise in the data but pose \eqref{eqn:CustPref} with a strongly concave utility function. Our design of this case study closely follows the scheme presented by \cite{Barmann2017}, but uses noisy data to learn the utility function coefficients. 

For each instance of the IOP, the training data are generated as follows. We first create an arbitrary utility vector $u \in \mathbb{R}^n$ by sampling its individual elements from the uniform distribution $\mathcal{U}(1, 1000)$, and normalize it to make its 1-norm equal to 1. We then sample a set of price vectors $w_i$, which are the input parameters for each $i \in \mathcal{I}$ such that $w_{ip} \sim \mathcal{U}(50, 150)$ for every $p \in \mathcal{P}$. The budget $b$ is set to $0.6 \sum\limits_{p \in \mathcal{P}} w_{1p}$ for all experiments. Next, keeping the utility vector the same, we solve these $|\mathcal{I}|$ instances of \eqref{eqn:CustPref} to obtain the optimal decisions $x^*_i$. We then generate the noisy datasets $\mathcal{J}_i$ for each $i \in \mathcal{I}$ by distorting the true optimal solution such that $x_{ij} = x^*_i + \gamma$, where $\gamma \sim \mathcal{N}(0, \sigma^2 \mathbb{I})$. 

In this study, we consider FOPs of varying dimensionality $n$ but limit the number of experiments to 100 in all cases. We also consider datasets with varying levels of noise by changing the value of $\sigma$. Once $n$ and $\sigma$ are fixed, the size of the sets $\mathcal{J}_i$ is kept the same for all $i \in \mathcal{I}$, i.e. $|\mathcal{J}_i| = J$ for all $i \in \mathcal{I}$. A specific case is hence represented by $n$, $\sigma$, and $J$, and we solve ten random instances of each case (generated using the scheme described above). Following Remark \ref{rem:high_noise}, to increase the robustness of our two-phase algorithm against such a large number of polyhedral geometries and different levels of noise, we introduce an additional set of constraints in the phase-1 problem enforcing $\hat{u}_p \geq 10^{-6}$ for all $p \in \mathcal{P}$. 

\subsubsection{Computational Performance}

The results comparing the computational performance of Algorithms \ref{alg:Two-phase} and \ref{alg:Online Algorithm} are summarized in Table \ref{tab:CustPref}. All instances were solved with a time limit of 7,200 s utilizing 24 cores on the Mesabi cluster of the Minnesota Supercomputing Institute (MSI). For each of the algorithms, the table lists the number of instances (out of ten) that were solved to optimality. In all instances that could not be solved with Algorithm \ref{alg:Two-phase}, we find that the solver could not find even a single feasible solution within the given time; hence the optimality gaps are not reported. For the decomposition-based algorithm, an unsolved instance is one where not all the 100 experiments could be processed in the given time but it still yields an estimate for the cost coefficients. Nonetheless, for the comparison of the two solution methods, we consider a run resulting in a partial solution as a failed run. Table \ref{tab:CustPref} shows the median, maximum, and minimum computation times for both algorithms. For Algorithm \ref{alg:Two-phase}, it is the time required to solve an instance without implementing the integer cuts in Phase 1. We find that adding the integer cuts to identify multiple optimal solutions was only required in 2 of the 180 $(\sim 1\%)$ instances considered in this study, confirming our assertion that a violation of Condition \ref{asp:uniqueX} is rare. Lastly, note that Algorithm \ref{alg:Online Algorithm} has an additional column labeled ``median $\#$ of re-solves" that shows the median value of the number of times the feasibility problem \ref{eqn:FP} was found infeasible, which then required the problem involving all experiments up to that point to be re-solved. 

From the data in Table \ref{tab:CustPref}, one can observe that the two algorithms solve the IOP in comparable times when the level of noise is low. However, as the size of the problem or the level of noise increases, Algorithm \ref{alg:Online Algorithm} starts outperforming Algorithm \ref{alg:Two-phase}. The difference in their performance is especially apparent from the numbers of instances solved and the maximum computation times. For example, in the arguably most difficult case with $n=100$, $\sigma=0.1$, and $J=20$, Algorithm 1 was not able to solve any of the given instances, while Algorithm \ref{alg:Online Algorithm} solved eight out of ten. Irrespective of the solution algorithm, increasing $J$ has a seemingly counterintuitive effect of making the problem easier to solve. This is because increasing the number of samples for an experiment merely increases the number of terms in the objective function of \eqref{eqn:P1} while making it easier for the problem to find the correct vertex to project onto. For Algorithm \ref{alg:Online Algorithm}, one can also see that increasing $J$ reduces the likelihood of incorrect projections when processing the data for individual experiments separately. This is especially helpful in situations where computing infrastructure is a limitation as large amount of data has a significantly higher memory requirement. Overall, the results indicate that Algorithm \ref{alg:Online Algorithm} dominates over Algorithm \ref{alg:Two-phase} when it comes to solving more difficult instances of \eqref{eqn:IOP} with data of high dimensionality and level of noise.

\begin{table}[h!]
\resizebox{\textwidth}{!}{%
\centering
\begin{tabular}{@{}cccccccccccc@{}}
\toprule
\multirow{3}{*}[-2em]{$n$} &
  \multirow{3}{*}[-2em]{$\sigma$} &
  \multirow{3}{*}[-2em]{$J$} &
  \multicolumn{4}{c}{Algorithm \ref{alg:Two-phase}} &
  \multicolumn{5}{c}{Algorithm \ref{alg:Online Algorithm}} \\ \cmidrule(lr){4-7} \cmidrule(lr){8-12} 
 &
   &
   &
  \multirow{2}{2cm}[-0.2em]{\centering $\#$ of instances solved} &
  \multicolumn{3}{c}{computation time (s)} &
  \multirow{2}{2cm}[-0.2em]{\centering $\#$ of instances solved} &
  \multicolumn{3}{c}{computation time (s)} &
  \multirow{2}{2cm}[-0.2em]{\centering median $\#$ of re-solves} \\         \addlinespace[3pt]
 \cmidrule(lr){5-7} \cmidrule(lr){9-11}
                     &                       &     &    & median & max & min &    & median & max & min &      \\ \addlinespace[3pt] \midrule
\multirow{6}{*}{25} & \multirow{2}{*}{0.01} & 5   & 10 & 33           & 155       & 6          & 10 & 277         & 317       & 253       & 2    \\
                     &                       & 250 & 10 & 73           & 90        & 67        & 10 & 266           & 271       & 264       & 0      \\
                     & \multirow{2}{*}{0.05} & 10  & 10 & 176          & 264       & 54        & 10 & 283          & 789       & 262       & 6    \\
                     &                       & 250 & 10 & 88           & 965       & 74        & 10 & 268          & 347       & 264       & 0    \\
                     & \multirow{2}{*}{0.1}  & 20  & 10 & 306          & 680       & 22        & 10 & 315          & 499       & 273       & 9.5  \\
                     &                       & 250 & 10 & 180          & 5,563      & 79        & 10 & 298             & 687       & 265       & 1 \\ \midrule
\multirow{6}{*}{50} & \multirow{2}{*}{0.01} & 5   & 10 & 88           & 931       & 7         & 10 & 484           & 638       & 331       & 1.5  \\
                     &                       & 250 & 10 & 220          & 1,724      & 200       & 10 & 423           & 717       & 414        & 0 \\
                     & \multirow{2}{*}{0.05} & 10  & 10 & 1,037         & 5,874      & 430       & 10 & 673           & 6,897      & 409      & 7.5  \\
                     &                       & 250 & 7  & 314         & 439       & 215       & 10 & 626          & 1,617      & 417       & 1.5  \\
                     & \multirow{2}{*}{0.1}  & 20  & 8  & 2,707            & 7,081      & 715       & 10 & 817          & 1,841      & 504       & 11.5 \\
                     &                       & 250 & 8  & 875          & 6,035      & 265       & 10 & 705          & 1,651      & 510       & 2 \\ \midrule
\multirow{6}{*}{100} & \multirow{2}{*}{0.01} & 5   & 9  & 407          & 488       & 29        & 10 & 428          & 1,042     & 362       & 2   \\
                     &                       & 250 & 10 & 489          & 2,267      & 419       & 10 & 479          & 1,648      & 451       & 0     \\
                     & \multirow{2}{*}{0.05} & 10  & 3  & 2,707         & 2,851      & 1,902       & 9  & 1,021         & 2,032      & 586       & 9    \\
                     &                       & 250 & 9  & 2,232         & 6,610      & 526       & 10 & 865           & 5,537      & 455       & 1.5   \\
                     & \multirow{2}{*}{0.1}  & 20  & 0  & n/a & n/a & n/a                        & 8  & 1,656         & 2,933      & 935       & 12    \\
                     &                       & 250 & 4  & 1,089         & 2,561      & 542       & 8  & 1,337            & 6,890      & 565       & 4.5 \\ \bottomrule
\end{tabular}}
\caption{\label{tab:CustPref}Comparison of computational performances of Algorithms \ref{alg:Two-phase} and \ref{alg:Online Algorithm} on an IOP based on random instances of \eqref{eqn:CustPref}. Reported computation times only consider the instances that were solved to optimality.}
\end{table}

\subsubsection{Prediction Error} \label{sec:CustPE}
Phase 2 of our two-phase algorithm requires a reference $\bar{u}$  to yield an estimate $\hat{u}$. Ideally, this reference is based on some prior intuition about the unknown utility function, but here we obtain $\hat{u}$ using a randomly generated $\bar{u}$. The goal is to test the capability of this estimate in generating reliable predictions on unseen datasets in the worst-case scenario where no prior information about the missing parameters is available. To perform this assessment, along with every instance of training data, we also generate a test dataset of 100 experiments. This test data consists of $(w, x^*)$ pairs, where $w$-values are generated in the same manner as for the training data, and $x^*$ are the corresponding true optimal solutions obtained using the same $u$ as the one used to generate the noisy training data. Once a $\hat{u}$ has been found, we use it to solve the problems in the test dataset and evaluate the prediction error as the fraction of incorrect predictions. 

We show the prediction error results for a few selected cases in Figure \ref{fig:CustLR}. Here, instead of using just the final estimate obtained with all the $|\mathcal{I}|$ experiments, we make use of the online algorithm (Algorithm 3 in {the online supplement, section B}) to show how the prediction error evolves with the addition of new experiments to the training set. As expected, the prediction error generally decreases with the number of experiments, and problems of higher dimensionality require more experiments to reach the same prediction accuracy. Also, the figures show an additional benefit of a large $J$ apart from making \eqref{eqn:P1} easier to solve. In Figures \ref{fig:CustLRa} and \ref{fig:CustLRb}, where $J$ is kept small, the curves are ``spiky" showing a local increase in prediction error. Recall that we solve \ref{eqn:FP} after every experiment $\ell$ to confirm if the the solution obtained with the first $\ell$ experiments $(\ell \leq |\mathcal{I}|)$ still holds for $\eqref{eqn:P1}_{[\ell]}$. Therefore, these local spikes are a consequence of insufficient sampling which results in the violation of our primary assumption that the vertex with minimum loss on the given data is the ``correct" vertex. As seen in Figure \ref{fig:CustLRc}, once an adequate number of samples are used, the likelihood of observing these peaks reduces significantly. We expect that in situations where the quality of data is uncertain and sampling is limited, adding a pre-processing step to remove large outliers from the dataset can be a way to prevent the model from making wrong estimates.

\begin{figure}[h!]
\centering
\subfloat[$\sigma = 0.01, J = 5$]{
    \label{fig:CustLRa}
    \includegraphics[width=0.33\textwidth]{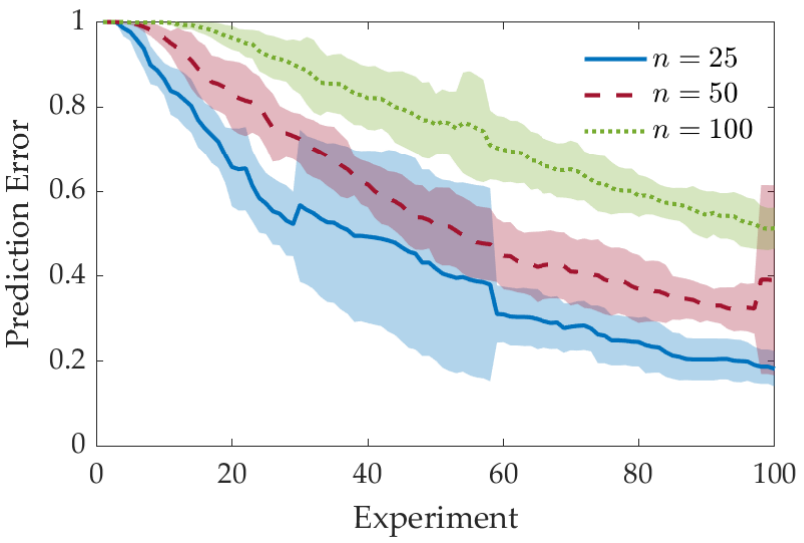}} 
\subfloat[$\sigma = 0.1, J = 20$]{
    \label{fig:CustLRb}
    \includegraphics[width=0.33\textwidth]{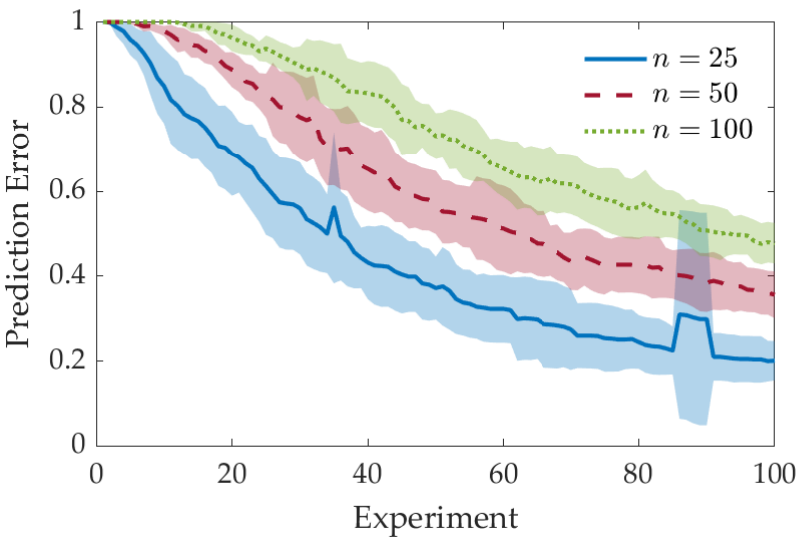}}
\subfloat[$\sigma = 0.1, J = 250$]{
    \label{fig:CustLRc}
    \includegraphics[width=0.33\textwidth]{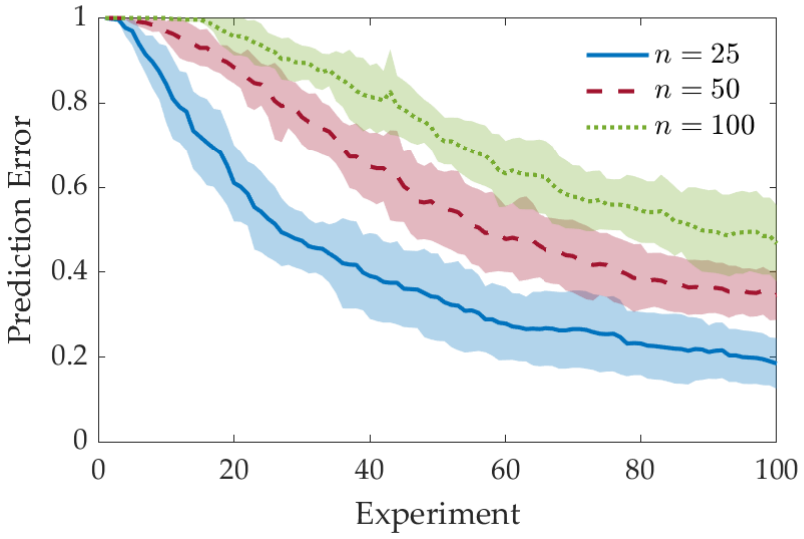}}
\caption{{(Color online)} \label{fig:CustLR}Change in prediction error as experiments are added. Lines show mean values across all solved instances, shaded areas indicate one standard deviation around the mean. For all the cases, prediction error data for individual instances are available in {the online supplement, section F}.}
\end{figure}

\subsubsection{Adaptive Sampling} \label{sec:CustAS}
We also apply our adaptive sampling strategy to this example. Following the data generation scheme used for random sampling, we define the set $\Pi$ as allowing any $w$ for which $50 \leq w_p \leq 150$ for any $p \in \mathcal{P}$ while keeping the other constraint parameters fixed. Irrespective of the size of \eqref{eqn:IOP}, we find that BARON struggles to find even a feasible solution for \eqref{eqn:ASP} when we try to solve it exactly. Therefore, we solve it here using the heuristic solution algorithm, Algorithm \ref{alg:RandomAdaIO}. All the $S$ subproblems were solved with a time limit of 100 s. Although BARON is unable to solve the subproblems to optimality, it still finds feasible solutions that show good potential in reducing the size of the admissible set. The significant impact of our adaptive sampling strategy is shown in Figure \ref{fig:CustAdaptive}. In all three cases, the use of adaptive sampling results in a more than 50\% reduction in the number of experiments required to achieve the same prediction accuracy as obtained from the standard approach using random sampling after 100 experiments. Furthermore, Figure \ref{fig:CustAdaptivea} shows the effect of the parameter $S$ in the heuristic solution approach for \eqref{eqn:ASP}. One can see that even a small $S$ results in a large increase in the rate of reduction of the prediction error and moreover, the rate stabilizes very quickly for a rather small $S$. By performing a similar study on higher-dimensional problems, we find that using an $S$ equal to the dimensionality of the problem $n$ to be a good heuristic to achieve the desired effect with adaptive sampling. 
\definecolor{Elecblue}{rgb}{0.45098,0.76078,0.9843}
\pgfplotsset{tick label style={font=\Large},
    label style={font=\Large},
    legend style={font=\large}}
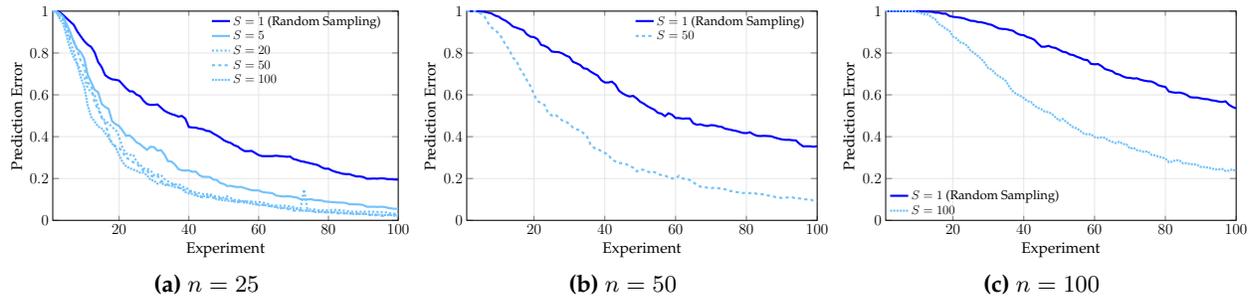
\begin{figure}[h!]
\centering
\subfloat[$n = 25$]{
    \label{fig:CustAdaptivea}
\begin{tikzpicture}[scale=0.4]
    \begin{axis}[xmin = 1, xmax = 100, ymin = 0, ymax = 1, ylabel = Prediction Error, xlabel = Experiment, xscale = 1.38, legend style={at={(0.7,1)},anchor=north east, draw=none, fill=none}, legend cell align={left}, y label style={at={(-0.055,0.5)}}, grid style={gray!20},ymajorgrids=true,xmajorgrids=true, height = 8.54cm]
		\addplot[blue, line width =  2 pt] table[x index = {0}, y index = {1}] {Adaptive25.txt};   \addlegendentry{$S = 1$ (Random Sampling)}
		\addplot[Elecblue, line width =  2 pt] table[x index = {0}, y index = {2}] {Adaptive25.txt};   \addlegendentry{$S = 5$}
		\addplot[Elecblue, dotted, line width =  2 pt] table[x index = {0}, y index = {3}] {Adaptive25.txt};   \addlegendentry{$S = 20$}
		\addplot[Elecblue, dashed, line width =  2 pt] table[x index = {0}, y index = {4}] {Adaptive25.txt};   \addlegendentry{$S = 50$}
		\addplot[Elecblue, densely dotted, line width =  2 pt] table[x index = {0}, y index = {5}] {Adaptive25.txt};   \addlegendentry{$S = 100$}
\end{axis}
\end{tikzpicture}}
\subfloat[$n = 50$]{
    \label{fig:CustAdaptiveb}
\begin{tikzpicture}[scale=0.4]
    \begin{axis}[xmin = 1, xmax = 100, ymin = 0, ymax = 1, ylabel = Prediction Error, xlabel = Experiment, xscale = 1.4, legend style={at={(0.7,1)},anchor=north east, draw=none, fill=none}, legend cell align={left}, y label style={at={(-0.055,0.5)}}, grid style={gray!20},ymajorgrids=true,xmajorgrids=true, height = 8.54cm]
	\addplot[blue, line width =  2 pt] table[x index = {0}, y index = {1}] {Adaptive50.txt}; \addlegendentry{$S = 1$ (Random Sampling)}
	\addplot[Elecblue, dashed, line width =  2 pt] table[x index = {0}, y index = {2}] {Adaptive50.txt}; \addlegendentry{$S = 50$}
\end{axis}
\end{tikzpicture}}
\subfloat[$n = 100$]{
    \label{fig:CustAdaptivec}
\begin{tikzpicture}[scale=0.4]
    \begin{axis}[xmin = 1, xmax = 100, ymin = 0, ymax = 1, ylabel = Prediction Error, xlabel = Experiment, xscale = 1.4, legend style={at={(0,0)},anchor=south west, draw=none, fill=none}, legend cell align={left}, y label style={at={(-0.055,0.5)}}, grid style={gray!20},ymajorgrids=true,xmajorgrids=true, height = 8.54cm]
	\addplot[blue, line width =  2 pt] table[x index = {0}, y index = {1}] {Adaptive100.txt}; \addlegendentry{$S = 1$ (Random Sampling)}
	\addplot[Elecblue, densely dotted, line width =  2 pt] table[x index = {0}, y index = {2}] {Adaptive100.txt}; \addlegendentry{$S = 100$}
\end{axis}
\end{tikzpicture}}
\caption{{(Color online)} \label{fig:CustAdaptive}Effect of adaptive sampling on the evolution of prediction error. Lines show mean values across ten instances. Training data for all cases was generated using $\sigma = 0.01, J = 30$.}
\end{figure}


\subsection{Cost Estimation for Production Planning} \label{sec:ProdPlan}
In our second case study, we consider the problem of production planning for a large manufacturing site consisting of multiple processes within an interconnected process network. It is commonly formulated as an LP:
\stepcounter{equation}
\begin{varsubequations}{13}
\label{eqn:ProdPlan}
    \begin{align}
     \minimize\limits_{x, y, w} \quad & \sum\limits_{h \in \mathcal{H}} \sum\limits_{m \in \mathcal{M}} \left( \sum\limits_{p \in \mathcal{P}} c_{pmh}x_{pmh} + f_{mh} w_{mh} \right) \label{eqn:MinCost} \\
     \st \quad & q^{\mathrm{min}}_{m} \leq q^0_m + \sum\limits_{h' = 1}^{h} \left( \sum\limits_{p \in \mathcal{P}} x_{pmh'} + w_{mh'} - d_{mh'}\right) \leq q^{\mathrm{max}}_m \quad \forall \, m \in \mathcal{M}, h \in \mathcal{H} \label{eqn:Inventory} \\
     & 0 \leq w_{mh} \leq w^{\mathrm{max}}_{mh} \quad \forall \, m \in \mathcal{M}, h \in \mathcal{H} \label{eqn:PurcRest} \\
     & x_{pmh} = \mu_{pm} y_{ph} \quad \forall \, p \in \mathcal{P}, m \in \mathcal{M}, h \in \mathcal{H} \label{eqn:NetDesign} \\
     & 0 \leq x_{pmh} \leq x^{\mathrm{max}}_{pmh} \quad \forall \, p \in \mathcal{P}, m \in \mathcal{M}, h \in \mathcal{H} \label{eqn:Bounds},
     \end{align}
\end{varsubequations}
where $\mathcal{P}, \mathcal{M}$, and $\mathcal{H} =  \{1, \ldots, H\}$ are the sets of processes, materials, and time periods, respectively. The amount of material $m$ produced or consumed (depending on the sign) by process $p$ in time period $h$ is denoted by $x_{pmh}$. Product demand and additional purchase of a material are denoted by $d_{mh}$ and $w_{mh}$, respectively. Inventory constraints and restrictions on the amounts purchased are stated in constraints \eqref{eqn:Inventory} and \eqref{eqn:PurcRest}, respectively. The structure of the process network is defined by constraints \eqref{eqn:NetDesign} where $y_{ph}$ denotes the amount of a reference material for process $p$ produced in time period $h$, and $\mu_{pm}$ is a conversion factor that specifies how much of a material $m$ is produced or consumed for one unit of the reference material. According to \eqref{eqn:MinCost}, the objective is to minimize the total production and purchasing cost while satisfying given product demand. While purchasing prices $f_{mh}$ are readily known, production costs $c_{pmh}$ are often difficult to estimate, which is one major reason why in practice, production planning is still mostly performed manually by human planners \citep{Troutt2006}. Experienced planners have an excellent intuition for the relative differences in costs, but this information is not explicitly expressed in numbers. The goal is to use past production plans, which reflect the planners' decisions, to infer these costs.

Here we conduct a case study using the process network for a petrochemical site with 28 chemicals and 38 processes from \citet{Sahinidis1989} and \citet{Zhang2016a}. The parameters required to model this network are given in {the online supplement, section G}. We test our methodology on this problem by generating synthetic training data simulating expert planners' decision making under different operating and market conditions. Here, each data point consists of inputs $(\mu, d)$ and the corresponding decisions $(x, y, w)$. For different scenarios, the conversion factors $\mu_{pm}$ have been assumed to vary between 75\% to 100\% of their nominal values on account of changing efficiencies of individual processes. Also, the product demand $d_{mh}$ is considered to fluctuate $\pm 10\%$ from its nominal value. For each instance, we generate 100 such scenarios, i.e. $|\mathcal{I}| = 100$. For each of these scenarios, we then generate the respective decisions by solving \eqref{eqn:ProdPlan} with the same arbitrarily generated $c_{pmh}$ and $f_{mh}$ values. Since human decision making is often inconsistent, we distort the true optimal solution $(x^*, y^*, w^*)$ as $(x^* + \gamma_1, y^* + \gamma_2, w^* + \gamma_3)$ where $\gamma_i \sim \mathcal{N}(0, \sigma^2)$ for all $i \in \{1, 2, 3\}$.

In this case study, we consider training data instances of three different sizes by varying the number of time periods $H$. All training data is generated using $\sigma = 3$ and $J = 30$. Each dataset is used to solve \eqref{eqn:IOP} with both Algorithm \ref{alg:Two-phase} and the decomposition-based Algorithm \ref{alg:Online Algorithm}. The problem instances were solved with a time limit of 14,400 s using 24 cores on the Mesabi cluster of the MSI. We find that while Algorithm \ref{alg:Two-phase} is unable to find even a feasible solution for any of the three cases, Algorithm \ref{alg:Online Algorithm} can solve these problems in less than 10 minutes.

Unlike the previous case study where we assessed the quality of only a point estimate obtained from Phase 2, here we instead focus on the quality of the set $\widehat{\mathcal{C}}$. To do this, we sample ten point estimates by solving \eqref{eqn:P2} with ten random reference cost vectors. We again consider a test dataset of 100 data points consisting of arbitrary $(\mu, d)$ and $(x^*, y^*, w^*)$ pairs. We relax the prediction error criteria to obtain a more realistic metric that measures the closeness of the generated predictions to the true optimal solutions; the metric is defined as $d_{\mathcal{V}}(x^*, \hat{x}) = \sum\limits_{v \in \mathcal{V}} \lVert x_v^* - \hat{x}_v \rVert_{\infty}$, where $\mathcal{V}$ is the set of test data points. Figure \ref{fig:ProdPlana} shows the change in the normalized distance metric with the addition of new experiments. One can verify that the spread around the mean value after 100 experiments is fairly small, implying high confidence in the final point estimate irrespective of the quality of the reference. However, this does not discount the importance of a good reference as the admissible set also likely contains the true $c$ which, if used as a reference, would result in a zero $d_{\mathcal{V}}$ even with a single experiment. 

\begin{figure}[h!]
    \centering
    \subfloat[Random sampling for all cases.]{
        \label{fig:ProdPlana}
        \includegraphics[width=0.45\textwidth]{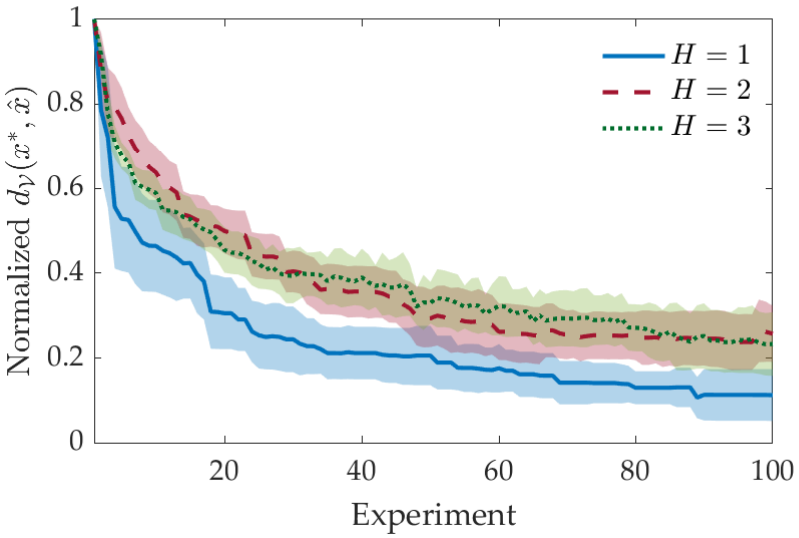}} \hspace{1em}
    \subfloat[Comparison between random and adaptive sampling for $H = 1$. For adaptive sampling, $S$ was set to 100.]{
        \label{fig:ProdPlanb}
        \includegraphics[width=0.45\textwidth]{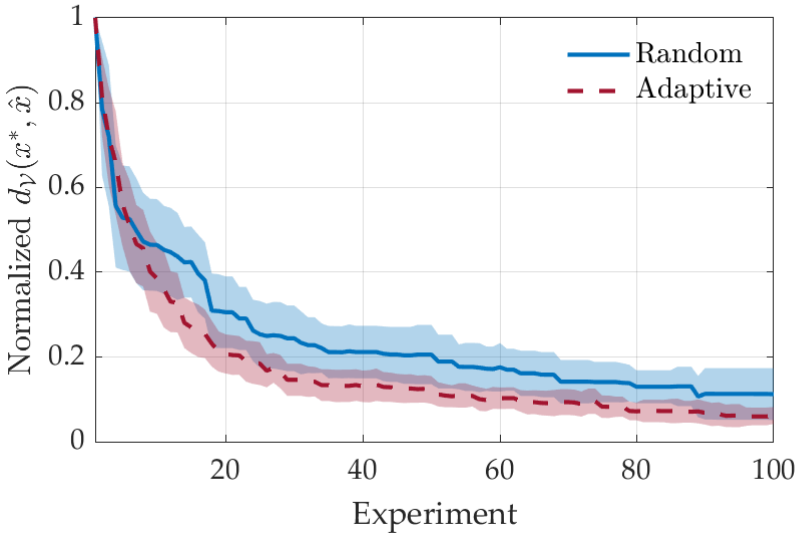}}
    \caption{{(Color online)} \label{fig:ProdPlan}Normalized $d_{\mathcal{V}}(x^*, \hat{x})$ as experiments are added to the training set. Lines show mean values across all instances, shaded areas indicate one standard deviation around the mean. Training data was generated for $\sigma = 3, J = 30.$}
\end{figure}

Notice that the curve representing the mean value decreases in discrete steps with several flat regions in between two steps. This is a consequence of not all inputs resulting in a reduction in the size of the admissible set. Therefore, we also evaluate the impact of adaptive sampling on mitigating this issue. Due to the larger size of the FOP here compared to \eqref{eqn:CustPref}, we solve the subproblems in Algorithm \ref{alg:RandomAdaIO} with an increased time limit of 200 s. However, even with the increased time limit, the solver struggles to find good feasible solutions for the larger instances with $H = 2$ (132 variables) and $H = 3$ (198 variables). Here, we show the effect of adaptive sampling with the case of $H = 1$ in Figure \ref{fig:ProdPlanb}. As can be observed, adaptive sampling sustains a high rate of decrease in the distance metric for a longer duration compared to naive random sampling. This results in just $\sim 52$ experiments being required to achieve the same effect as 100 random inputs. Moreover, one can see that the variance around the mean in the case of adaptive sampling is noticeably lower which shows that in addition to requiring less experiments, it also finds estimates with higher confidence levels.

\section{Conclusions} \label{sec:Conc}
In this work, we addressed data-driven inverse linear optimization with noisy observations, for which we introduced a new problem formulation that offers two practical advantages over other existing methods: (i) It allows the recovery of a less restrictive and generally more appropriate admissible set of cost estimates by assuming that the optimal solutions of the FOP lie at the vertices of the feasible region. (ii) Instead of randomly choosing a point estimate from the admissible set, it makes use of a reference cost vector to choose the cost estimate that most resembles the user's prior belief.

An exact two-phase algorithm was developed to solve the IOP, and we further proposed an efficient decomposition algorithm and an adaptive sampling method that are especially suited for an online inverse optimization setting. Results from extensive computational experiments based on two case studies show that the proposed methods are effective in significantly reducing both the computation time and data requirement for generating cost estimates with a reasonably low prediction error on unseen datasets.

\section*{Acknowledgments}
The authors gratefully acknowledge the financial support from the National Science Foundation under Grant No. 2044077 and the Minnesota Supercomputing Institute (MSI) at the University of Minnesota for providing resources that contributed to the research results reported in this paper. RG acknowledges financial support from a departmental fellowship sponsored by 3M.

\bibliographystyle{newapa}
\bibliography{Inverse_Optimization}

\newpage
\setcounter{page}{1}

\centerline{\LARGE{\textbf{Online Supplement}}}

\renewcommand*{\thesection}{\Alph{section}.}
\setcounter{section}{0}


\section{Omitted Proofs} \label{sec:proofs}
\begin{manuallemma}{1}[]
\textit{The set $\widehat{\mathcal{C}}'$ is nonempty.}
\end{manuallemma}
\begin{proof}
Since $\{\tilde{x}: A_i \tilde{x} \leq b_i\}$ is assumed to be compact and nonempty, $\argmin_{\tilde{x} \in \mathbb{R}^n} \left\lbrace \hat{c}^{\top} \tilde{x}: A_i \tilde{x} \leq b_i \right\rbrace$ is nonempty, and for every $\hat{c} \in \mathbb{R}^n$, there is at least one vertex of $\{\tilde{x}: A_i \tilde{x} \leq b_i\}$ that minimizes $\hat{c}^{\top} \tilde{x}$. This implies that $\argmin_{\tilde{x} \in \mathbb{R}^n} \left\lbrace \hat{c}^{\top} \tilde{x}: A_i \tilde{x} \leq b_i \right\rbrace \cap \mathcal{V}_i$ is nonempty and finite, due to the finiteness of $\mathcal{V}_i$, for all $i \in \mathcal{I}$ and $\hat{c} \in \mathbb{R}^n$. As a result, there is at least one $(\hat{c}, \hat{x})$ that minimizes the loss function in the definition of the outer $\argmin$ set in \eqref{eqn:Vertex_set}; hence, $\widehat{\mathcal{C}}'$ is nonempty.
\end{proof}

\begin{manualcorollary}{1}[]
\textit{If Condition \ref{asp:uniqueX} holds, \eqref{eqn:IOP} has a unique solution.}
\end{manualcorollary}
\begin{proof}
Condition \ref{asp:uniqueX} implies that $|\widehat{\mathcal{X}}^*| = 1$. Corollary \ref{cor:uniqueSolution} then directly follows from Theorem \ref{thm:finiteSolutions}.
\end{proof}

\begin{manuallemma}{2}[]
\textit{Problems \eqref{eqn:IOP_innerMin} and \eqref{eqn:P1} have the same set of feasible $\hat{x}$.}
\end{manuallemma}

\begin{proof}
By construction, for a given feasible $\hat{c}$, \eqref{eqn:IOP_innerMin} and \eqref{eqn:P1} exhibit the same set of feasible $\hat{x}$. Observe that any $\hat{c}$ for which $\lVert \hat{c} \rVert_1 \neq 1$ is infeasible in \eqref{eqn:P1}. Let $\tilde{c}$ denote a $\hat{c}$ that is feasible in \eqref{eqn:IOP_innerMin} but not in \eqref{eqn:P1}, i.e. $\lVert \tilde{c} \rVert_1 \neq 1$. However, there exists a positive $\alpha$ such that $\hat{c} = \alpha \tilde{c}$ is feasible in \eqref{eqn:P1}, and it exhibits the same set of feasible $\hat{x}$ as $\tilde{c}$ does in \eqref{eqn:IOP_innerMin} since uniform positive scaling of a cost vector does not change the set of optimal solutions of an LP. Hence, problems \eqref{eqn:IOP_innerMin} and \eqref{eqn:P1} have the same set of feasible $\hat{x}$.
\end{proof}

\begin{manualproposition}{1}[]
\textit{If Condition \ref{asp:uniqueX} holds, the optimal solution to \eqref{eqn:IOP} can be obtained by solving \eqref{eqn:P1}, which provides the unique optimal $\hat{x}^*$, and subsequently solving \eqref{eqn:P2} with $\hat{x} = \hat{x}^*$.
}\end{manualproposition}

\begin{proof}
Proposition \ref{prp:TwoPhase} directly follows from Corollary \ref{cor:uniqueSolution} and Lemma \ref{lem:epsilon}.
\end{proof}

\begin{manuallemma}{3}[]
\textit{Given a set of experiments $\{1, \ldots, N\}$, let $\eqref{eqn:P1}_i$ with $i \leq N$ be an instance of \eqref{eqn:P1} with $\mathcal{I} = \{ i \}$ and let its optimal value be $\bar{r}_i^*$. If $r^*$ denotes the optimal value of \eqref{eqn:P1} with $\mathcal{I} = \{ 1, \dots, N\}$, then $\sum\limits_{i=1}^N \bar{r}_i^* \leq r^*$.
}\end{manuallemma}

\begin{proof}
Consider the following reformulation of \eqref{eqn:P1}:
\begin{equation}
\label{eqn:P1'}
\tag{P1'}
\begin{aligned}
    \minimize\limits_{\hat{c}, \, \hat{x}, \, s, \, \lambda, \, z, \, w, \, \hat{c}^+, \hat{c}^-} \quad & \sum\limits_{i \in \mathcal{I}}\sum\limits_{j \in \mathcal{J}_{i}} \lVert x_{ij} - \hat{x}_{i} \rVert \\
    \st \quad \quad & \hat{c}_i = \hat{c}_{i+1} \quad \forall \, i \in \{1, \ldots, N-1\} \\
    & \hat{c}_i + A_i^{\top} \lambda_{i} = 0 \quad \forall \, i \in \mathcal{I}  \\
    & A_i\hat{x}_{i} + s_{i} = b_i \quad \forall \, i \in \mathcal{I}  \\
    & \lambda_{i} \leq M z_{i} \quad \forall \, i \in \mathcal{I} \\
    & s_{i} \leq M (e-z_{i}) \quad \forall \, i \in \mathcal{I}   \\
    & e^{\top} z_i \geq n \quad \forall \, i \in \mathcal{I}  \\
    & \hat{c}_i = \hat{c}_i^+ - \hat{c}_i^-  \quad \forall \, i \in \mathcal{I} \\
    & \hat{c}_i^+ \leq w_i  \quad \forall \, i \in \mathcal{I} \\
    & \hat{c}_i^- \leq e-w_i \quad \forall \, i \in \mathcal{I} \\
    & e^{\top} (\hat{c}_i^+ + \hat{c}_i^-) = 1   \quad \forall \, i \in \mathcal{I} \\
    & \hat{x}_i \in \mathbb{R}^n, \, s_i \in \mathbb{R}_+^m, \, \lambda_i \in \mathbb{R}_+^m, \, z_i \in \{0,1\}^m \quad \forall \, i \in \mathcal{I} \\
    & \hat{c}_i \in \mathbb{R}^n, \, \hat{c}_i^+ \in \mathbb{R}^n_+ , \, \hat{c}_i^- \in \mathbb{R}^n_+, \, w_i \in \{0,1\}^n \quad \forall \, i \in \mathcal{I},
\end{aligned}
\end{equation}
where we replace the variable $\hat{c}$ in \eqref{eqn:P1} with $N$ disaggregated variables denoted by $\hat{c}_i$ for each experiment $i \in \mathcal{I}$. We further add the first set of constraints, which equate all $\hat{c}_i$ to ensure \eqref{eqn:P1'} is equivalent to \eqref{eqn:P1}. Consider a relaxation of \eqref{eqn:P1'} obtained by removing these first set of constraints. With this change, the problem decomposes into $N$ independent subproblems, one for each experiment $i \in \mathcal{I}$, which is precisely $\eqref{eqn:P1}_i$. It follows that the optimal value of the relaxed problem is $\sum\limits_{i=1}^N \bar{r}_i^*$, and since it is a relaxation of \eqref{eqn:P1'} and therefore also a relaxation of \eqref{eqn:P1},  $\sum\limits_{i=1}^N \bar{r}_i^* \leq r^*$.
\end{proof}

\begin{manualcorollary}{2}[]
\textit{Let $\eqref{eqn:P1}_{[i]}$ be an instance of \eqref{eqn:P1} with $\mathcal{I} = \{1, \ldots, i\}$ and $\bar{r}^*_{[i]}$ be its optimal value. Then, $\bar{r}^*_{[i-1]} + \bar{r}^*_{i} \leq \bar{r}^*_{[i]}$.
}\end{manualcorollary}
\begin{proof}
Corollary \ref{cor: SequentialUpdate} follows from Lemma \ref{lem:valLB}.
\end{proof}

\begin{manualproposition}{2}[]
\textit{Let $(\hat{x}_1, \ldots, \hat{x}_{\ell-1})$ and $\hat{x}_{\ell}$ be feasible solutions to problems $\eqref{eqn:P1}_{[\ell-1]}$ and $\eqref{eqn:P1}_{\ell}$, respectively. Then, $(\hat{x}_1, \ldots, \hat{x}_{\ell})$ is a feasible solution to $\eqref{eqn:P1}_{[\ell]}$ if the following problem is feasible:
}
\begin{equation}
    \label{eqn:FP}
    \tag*{$(\mathrm{FP})_{\ell}$}
    \begin{aligned}
        \minimize\limits_{\hat{c}, \hat{c}^+, \hat{c}^-, \gamma, w} \quad & 0 \\
        \st \quad & \hat{c}  = \sum\limits_{t \in \mathcal{T}(\hat{x}_i)} \gamma_{it}(-a_{it}^{\top}) \quad \forall \, i \in \{1, \ldots, \ell \} \\
        & {\hat{c}} = {\hat{c}}^+ - {\hat{c}}^-  \\
        & {\hat{c}}^+ \leq w  \\
        & {\hat{c}}^- \leq e-w \\
        & e^{\top} ({\hat{c}}^+ + {\hat{c}}^-) = 1   \\        
        & \gamma_{it} \geq 0 \quad \forall \, i \in \{1, \ldots, \ell \}, t \in \mathcal{T}(\hat{x}_i) \\
        & \hat{c} \in \mathbb{R}^n,\,{\hat{c}}^+ \in \mathbb{R}^n_+,\,{\hat{c}}^- \in \mathbb{R}^n_+ ,\, w \in \{0,1\}^n.
    \end{aligned}
\end{equation}
\end{manualproposition}
\begin{proof}
By construction, a feasible \ref{eqn:FP} implies that $\{\hat{x}_1, \ldots, \hat{x}_{\ell} \}$ are such that there exists a $\hat{c}$ in $\bigcap\limits_{i \in \mathcal{I}} \mathrm{cone}\left(\{-a_{it}\}_{t \in \mathcal{T}(\hat{x}_{i})}\right)$ with $\lVert \hat{c} \rVert_1 = 1$. This indicates that for the given solutions to $\eqref{eqn:P1}_{[\ell-1]}$ and $\eqref{eqn:P1}_{\ell}$, one can find a common $\hat{c}$ that is feasible in both problems, which in turn implies that $(\hat{x}_1, \ldots, \hat{x}_{\ell})$ is feasible in $\eqref{eqn:P1}_{[\ell]}$.
\end{proof}

\begin{manualcorollary}{3}[]
\textit{Let $(\hat{x}_1, \ldots, \hat{x}_{\ell-1})$ and $\hat{x}_{\ell}$ be optimal solutions to problems $\eqref{eqn:P1}_{[\ell-1]}$ and $\eqref{eqn:P1}_{\ell}$, respectively. Then, $(\hat{x}_1, \ldots, \hat{x}_{\ell})$ is an optimal solution to $\eqref{eqn:P1}_{[\ell]}$ if \ref{eqn:FP} is feasible.
}\end{manualcorollary}
\begin{proof}
Due to Corollary \ref{cor: SequentialUpdate}, the sum of the optimal values of $\eqref{eqn:P1}_{[\ell-1]}$ and $\eqref{eqn:P1}_{\ell}$ provide a lower bound on the optimal value of $\eqref{eqn:P1}_{[\ell]}$, i.e. $\bar{r}^*_{[\ell-1]} + \bar{r}^*_{\ell} \leq \bar{r}^*_{[\ell]}$. If \ref{eqn:FP} is feasible, $(\hat{x}_1, \ldots, \hat{x}_{\ell})$ is feasible in $\eqref{eqn:P1}_{[\ell]}$ due to Proposition \ref{prp:DecompValidity} and therefore yields an upper bound, i.e. $\bar{r}^*_{[\ell-1]} + \bar{r}^*_{\ell} \geq \bar{r}^*_{[\ell]}$. As the lower and upper bounds coincide, it follows that $(\hat{x}_1, \ldots, \hat{x}_{\ell})$ is an optimal solution to $\eqref{eqn:P1}_{[\ell]}$.
\end{proof}

\newpage
\section{Algorithm Pseudocode} \label{sec:OnlineExt}

\setcounter{algorithm}{2}
\begin{algorithm}
\begin{algorithmic}[1]
    \State initialize: $\widehat{\mathcal{C}} \gets \mathbb{R}^n$
    \ForAll{$ \ell \in \{1, \dots, N\}$}
        \State solve $\eqref{eqn:P1}_{\ell}$, obtain $\hat{x}_{\ell}^*$
        \State $\widehat{\mathcal{C}} \gets \widehat{\mathcal{C}} \, \bigcap \, \mathrm{cone}\left(\{-a_{\ell t}\}_{t \in \mathcal{T}(\hat{x}^*_{\ell})}\right)$
        \State solve \ref{eqn:FP}
        \If {\ref{eqn:FP} is feasible}
            \State solve \eqref{eqn:P2} with $\mathcal{I} = \{1,\ldots,\ell \}$, obtain $\hat{c}_{\ell}^*$
        \Else
            \State Solve (P1)$_{[\ell]}$, warm-start with $(\hat{x}_1^*, \ldots, \hat{x}_{\ell - 1}^*)$, obtain $\hat{x}_{[\ell]}^* = (\hat{x}_1^*, \ldots, \hat{x}_{\ell}^*)$
            \State $\widehat{\mathcal{C}} \gets \bigcap\limits_{i \in \{1, \ldots, \ell \}} \mathrm{cone}\left(\{-a_{it}\}_{t \in \mathcal{T}(\hat{x}^*_{i})}\right)$
            \State solve \eqref{eqn:P2} with $\mathcal{I} = \{1,\ldots,\ell \}$, obtain $\hat{c}_{\ell}^*$
        \EndIf
    \EndFor
        \State $\hat{c}^* \gets \hat{c}^*_{|\mathcal{I}|}$
    \caption{An efficient algorithm for solving \eqref{eqn:IOP} in an online environment}
    \label{alg:Online Algorithm_1}
\end{algorithmic}
\end{algorithm}


\section{Example Illustrating Remark \ref{rem:high_noise}} \label{sec:high_noise_examp}
Let the \eqref{eqn:FOP} be  $\max\limits_{x_1, x_2} \{2x_1 + x_2 \, : \, 14x_1 + 10x_2 \leq 17, 0 \leq x_1, x_2 \leq 1\}$.  Let $|\mathcal{I}| = 1$ and  $|\mathcal{J}| = 4$ with $x_1 = [(1.05, 0.26), (0.9, 0.15), (1.2, 0.15), (1, 0.03)].$. Clearly, the true solution is $(1,0.3)$. However, \eqref{eqn:P1} yields the denoised estimate $\hat{x}_1 = (1, 0)$, whereas, if we consider additional constraints stating $\hat{c} < 0$ (note that the FOP here is a maximization problem), the estimate becomes $\hat{x}_1 = (1, 0.3)$. These results are illustrated in Figure \ref{fig:HighNoise}. 

\pgfplotsset{every tick label/.append style={font=\small}}
\definecolor{dtsfsf}{rgb}{1,0,0}
\definecolor{sexdts}{rgb}{0.1803921568627451,0.65,0.19607843137254902}
\definecolor{rvwvcq}{rgb}{0.08235294117647059,0.396078431372549,0.7529411764705882}
\begin{figure}[h]
    \centering
    \begin{minipage}{.4\textwidth}
        \begin{tikzpicture}[line cap=round,line join=round,>=triangle 45,x=4cm,y=4cm]
            \begin{axis}[
                x=4cm,y=4cm,
                axis lines=middle,
                grid style=dashed,
                ymajorgrids=true,
                xmajorgrids=true,
                xmin=0,
                xmax=1.3,
                ymin=-0,
                ymax=1.2,
                xlabel={$x_1$}, ylabel={$x_2$},
                xlabel style={at={(.95,-0.13)}, anchor=south},
                ylabel style={at={(-0.13,0.95)}, anchor=west},
                xtick={0,0.2,...,1},
                ytick={0,0.2,...,1},]
                \clip(-0.6700162530465679,-0.3202016250014847) rectangle (1.481873185705448,1.1129930954148999);
                \fill[line width=2pt,color=rvwvcq,fill=rvwvcq,fill opacity=0.15] (0,1) -- (0.5,1) -- (1,0.3) -- (1,0) -- (0,0) -- cycle;
                \draw [line width=2pt,color=rvwvcq] (0,1)-- (0.5,1);
                \draw [line width=2pt,color=rvwvcq] (0.5,1)-- (1,0.3);
                \draw [line width=2pt,color=rvwvcq] (1,0.3)-- (1,0);
                \draw [line width=2pt,color=rvwvcq] (1,0)-- (0,0);
                \draw [line width=2pt,color=rvwvcq] (0,0)-- (0,1);
            \end{axis}
            \draw [line width = 1pt,color=dtsfsf] (1.05,0.26) circle (4pt);
            \draw [line width = 1pt,color=dtsfsf] (0.9,0.15) circle (4pt);
            \draw [line width = 1pt,color=dtsfsf] (1.1,0.15) circle (4pt);
            \draw [line width = 1pt,color=dtsfsf] (1,0.03) circle (4pt);
            \draw [line width = 1pt,color=sexdts] (1,0) circle (4pt);
            \fill[color=sexdts,fill=sexdts,fill opacity=1] (1,0) circle (4 pt);
        \end{tikzpicture}
        \caption*{(a) Problem \eqref{eqn:P1} solved without $\hat{c} < 0$}
    \end{minipage}
    \begin{minipage}{.4\textwidth}
        \begin{tikzpicture}[line cap=round,line join=round,>=triangle 45,x=4cm,y=4cm]
            \begin{axis}[
                x=4cm,y=4cm,
                axis lines=middle,
                grid style=dashed,
                ymajorgrids=true,
                xmajorgrids=true,
                xmin=0,
                xmax=1.3,
                ymin=0,
                ymax=1.2,
                xlabel={$x_1$}, ylabel={$x_2$},
                xlabel style={at={(.95,-0.13)}, anchor=south},
                ylabel style={at={(-0.13,0.95)}, anchor=west},
                xtick={0,0.2,...,1},
                ytick={0,0.2,...,1},]
                \clip(-0.6700162530465679,-0.3202016250014847) rectangle (1.481873185705448,1.1129930954148999);
                \fill[line width=2pt,color=rvwvcq,fill=rvwvcq,fill opacity=0.15] (0,1) -- (0.5,1) -- (1,0.3) -- (1,0) -- (0,0) -- cycle;
                \draw [line width=2pt,color=rvwvcq] (0,1)-- (0.5,1);
                \draw [line width=2pt,color=rvwvcq] (0.5,1)-- (1,0.3);
                \draw [line width=2pt,color=rvwvcq] (1,0.3)-- (1,0);
                \draw [line width=2pt,color=rvwvcq] (1,0)-- (0,0);
                \draw [line width=2pt,color=rvwvcq] (0,0)-- (0,1);
            \end{axis}
            \draw [line width = 1pt,color=sexdts] (1,0.3) circle (4pt);
            \draw [line width = 1pt,color=dtsfsf] (1.05,0.26) circle (4pt);
            \draw [line width = 1pt,color=dtsfsf] (0.9,0.15) circle (4pt);
            \draw [line width = 1pt,color=dtsfsf] (1.1,0.15) circle (4pt);
            \draw [line width = 1pt,color=dtsfsf] (1,0.03) circle (4pt);
            \fill[color=sexdts,fill=sexdts,fill opacity=1] (1,0.3) circle (4 pt);
        \end{tikzpicture}
    \caption*{(b) Problem \eqref{eqn:P1} solved with $\hat{c} < 0$}
    \end{minipage}
    \caption{The blue shaded regions represent the feasible region of the FOP described in \ref{sec:high_noise_examp} Red hollow circles denote the noisy observations, whereas the denoised estimates are depicted by green circles.}
    \label{fig:HighNoise}
\end{figure}

\section{Reformulation of \eqref{eqn:ASP}} \label{sec:ASPRefor}
In this section, we present reformulations of \eqref{eqn:ASP} that can be solved using off-the-shelf solvers. We start by reformulating \eqref{eqn:ASP} as follows:
\stepcounter{equation}
\begin{varsubequations}{14}
\label{eqn:AdaRefor1}
\begin{align}
    \maximize\limits_{\eta, y, \gamma, z, A_{\ell}, b_{\ell}, \lambda, \hat{x}_{\ell}, s} \quad & \eta \\
    \st \quad \quad & y = \sum_{t \in \mathcal{T}(\hat{x}_i)} \gamma_{it} \left(- \frac{a_{it}}{\lVert a_{it} \rVert_2}\right) \quad \forall \, i \in \mathcal{I} \label{eqn:RepeatConst}  \\
    &  0 \leq  \gamma_{it} \leq \epsilon \quad \forall \, i \in \mathcal{I}, t \in \mathcal{T}(\hat{x}_i) \\
    & (A_{\ell}, b_{\ell}) \in \Pi  \\
    & \tilde{c}_{\ell-1} + A_{\ell}^{\top} \lambda = 0  \\
    & A_{\ell}\hat{x}_{\ell} + s = b_{\ell}  \label{eqn:Inner1Start} \\
    & \lambda \leq M z \\
    & s \leq M (e-z)   \\
    & e^{\top} z \geq n \\
    & \hat{x}_{\ell} \in \mathbb{R}^n, \, y \in \mathbb{R}^n \\
    &  s_k \geq 0, \, \lambda_k \geq 0, \, z_k \geq 0 \quad \forall \, k \in \mathcal{K} \label{eqn:Inner1end} \\
    & \eta = \min\limits_{\bar{\gamma}, \delta} \; e^{\top}\delta \\
    & \begin{aligned} \label{eqn:AdaInnerConst}
      \st \quad \;\; & -\delta \leq y - \sum\limits_{k \in \mathcal{K}} \bar{\gamma}_k z_k a_{\ell k} \\ 
        & y - \sum\limits_{k \in \mathcal{K}} \bar{\gamma}_k z_k a_{\ell k} \leq \delta \\
        & \delta \in \mathbb{R}^n_+ \\
        & \bar{\gamma}_{k} \geq 0 \quad \forall \,  k \in \mathcal{K}, \\
    \end{aligned}
\end{align}
\end{varsubequations}
where constraints \eqref{eqn:Inner1Start} - \eqref{eqn:Inner1end} are the KKT optimality conditions of the linear optimization problem embedded in constraint set \eqref{eqn:ASPConst5}. Here, we exploit the fact that the binary vector $z$ occurs naturally in the linearized version of the KKT conditions. We also linearize the optimization problem embedded in the objective function of \eqref{eqn:ASP}. Since the sole source of nonlinearity in this inner problem is the norm operator, we introduce a variable $\delta$ to linearize it. The rest of the constraints remain the same as in \eqref{eqn:ASP}.

Now, since \eqref{eqn:AdaRefor1} is a bilevel optimization problem with an LP embedded in it, the problem can be reformulated by replacing the lower-level problem with its optimality conditions, based on strong duality or KKT. Here, we present both reformulations: 

\begin{enumerate}
\item Reformulating \eqref{eqn:AdaRefor1} into a single-level problem by using strong duality for the lower-level problem:
\stepcounter{equation}
\begin{varsubequations}{15}
\label{eqn:AdaDuality}
\begin{align}
    \maximize\limits_{\eta, y, \gamma, \bar{\gamma}, z, A_{\ell}, b_{\ell}, \lambda, \hat{x}, s, \delta, \nu, \omega} \quad & \eta \\
    \st \quad \;\;\;\;\;\;\;\, & \eqref{eqn:RepeatConst} - \eqref{eqn:Inner1end} \\
    & \eta = (\nu - \omega)^{\top} y \label{eqn:StDual}\\
    & -\delta \leq y - \sum\limits_{k \in \mathcal{K}} \bar{\gamma}_k z_k a_{\ell k} \label{eqn:PrFeas1}\\
    & y - \sum\limits_{k \in \mathcal{K}} \bar{\gamma}_k z_k a_{\ell k} \leq \delta \label{eqn:PrFeas2} \\
    & \nu + \omega \leq e \label{eqn:DualFeas1} \\
    & (\nu  - \omega)^{\top}a_{\ell k} z_k  \leq 0 \quad \forall \, k \in \mathcal{K} \label{eqn:DualFeas2} \\
    & \delta \in \mathbb{R}_+^n, \; \bar{\gamma}_k \geq 0 \quad \forall \, k \in \mathcal{K} \\
    & \nu \in \mathbb{R}_+^n, \; \omega \in \mathbb{R}_+^n,
\end{align}
\end{varsubequations}
where $\nu$ and $\omega$ are the dual variables for the lower-level problem in \eqref{eqn:AdaRefor1}. Constraint \eqref{eqn:StDual} imposes the strong duality condition by equating primal and dual objective functions of the lower-level problem in \eqref{eqn:AdaRefor1}. Constraints \eqref{eqn:PrFeas1}-\eqref{eqn:PrFeas2} are the primal feasibility conditions, \eqref{eqn:DualFeas1}-\eqref{eqn:DualFeas2} are the dual feasibility conditions.

\item Reformulating \eqref{eqn:AdaRefor1} into a single-level problem by using KKT-based optimality conditions for the lower-level problem:
\stepcounter{equation}
\begin{varsubequations}{16}
\label{eqn:AdaKKT}
\begin{align}
    \maximize\limits_{\substack{\eta, y, \gamma, \bar{\gamma}, z, A_{\ell}, b_{\ell}, \lambda, \hat{x}, s, \delta, \\ \nu, \omega, \psi, \phi, s^1, s^2, z^1, z^2, z^3, z^4}} \quad & \eta \\
    \st \quad \;\;\;\;\;\;\; & \eqref{eqn:RepeatConst} - \eqref{eqn:Inner1end} \\
    & e - \nu - \omega - \psi = 0 \label{eqn:KKTStatCond1}\\
    & (\nu  - \omega)^{\top}a_{\ell k} z_k  - \phi_k = 0 \quad \forall \, k \in \mathcal{K} \label{eqn:KKTStatCondn2}\\
    & -\delta + s^1 = y - \sum\limits_{k \in \mathcal{K}} \bar{\gamma}_k z_k a_{\ell k} \label{eqn:KKTPrFeas1}\\
    & y - \sum\limits_{k \in \mathcal{K}_{M}} \bar{\gamma}_k z_k a_{\ell k} + s^2 = \delta \label{eqn:KKTPrFeas2} \\
    & s^1 \leq Mz^1 \label{eqn:KKTCompSl1} \\
    & \nu \leq M(e-z^1) \\
    & s^2 \leq Mz^2 \\
    & \omega \leq M(e-z^2) \\
    & \delta \leq Mz^3 \\
    & \psi \leq M(e-z^3) \\
    & \bar{\gamma} \leq Mz^4 \\
    & \phi \leq M(e-z^4) \label{eqn:KKTCompSl2} \\
    & \delta \in \mathbb{R}^n_+, \, \nu \in \mathbb{R}^n_+, \, \omega \in \mathbb{R}^n_+, \, \psi \in \mathbb{R}^n_+, s^1 \in \mathbb{R}^n_+, \, s^2 \in \mathbb{R}^n_+ \\
    & z^1 \in \{0, 1\}^n, \, z^2 \in \{0, 1\}^n, \, z^3 \in \{0, 1\}^n \\
    & \bar{\gamma}_k \geq 0, \, \phi_k \geq 0, \, z^4_k \in \{0, 1\}  \quad \forall \, k \in \mathcal{K},
\end{align}
\end{varsubequations}
where $\nu, \, \omega, \, \psi,$ and $\phi$ are the Lagrange multipliers of the constraints \eqref{eqn:AdaInnerConst} of the lower-level problem in \eqref{eqn:AdaRefor1}. Constraints \eqref{eqn:KKTStatCond1} - \eqref{eqn:KKTStatCondn2} are the stationarity conditions, \eqref{eqn:KKTPrFeas1} - \eqref{eqn:KKTPrFeas2} are the primal feasibility conditions, and \eqref{eqn:KKTCompSl1} - \eqref{eqn:KKTCompSl2} represent a linearized version of the complementary conditions corresponding to the constraints \eqref{eqn:AdaInnerConst} of the lower-level problem in \eqref{eqn:AdaRefor1}.
\end{enumerate}
In our computational case studies, we find that the KKT-based reformulation \eqref{eqn:AdaKKT} results in a relatively loose formulation owing to a large number of big-M parameters. Therefore, we employ the strong duality-based \eqref{eqn:AdaDuality} to choose input parameters for the case studies in Section \ref{sec:CompSt}.

\section{Effect of Restrictions on the Design of Experiments} \label{sec:AdSampDiff}
Problem \eqref{eqn:ASP} is valid regardless of the restrictions on the input parameters $A$ and $b$, i.e. the set $\Pi$. However, it is still worth highlighting the difference between being able to vary the constraint matrix $A$ versus being restricted to only varying the right-hand-side (RHS) parameters $b$. While the first case is more general (we include two case studies in the next section), the second scenario is also not uncommon. In several transportation network flow problems, the values in $A$ are restricted to 0 and 1. These values determine the structure of the network and cannot be changed. The RHS parameters $b$, however, usually refer to quantities such as capacities of individual arcs or demands at the nodes and can often be controlled. Notice that the primary requirement for adaptive or even random sampling to reduce the size of $\widehat{\mathcal{C}}$ is to find cones in successive experiments with extreme rays pointing in directions different from the ones already found. With $A$ being allowed to vary, this can be easily achieved by appropriately varying the parameters of the active constraints under the given restrictions $\Pi$. However, when only $b$ can vary, the only hope to reduce the size of the admissible set is in finding a set of input parameters for which a different set of constraints become active. We further illustrate this point through the following example.
\begin{example} \label{examp:varyingb}
Let the \eqref{eqn:FOP} for the first experiment be $\max\limits_{x_1, x_2} \{x_1 + 2x_2 \, : \, 2x_1 + 3x_2 \leq 3, x_1 \leq 1, x_2 \leq 1, x_1 \geq 0, x_2 \geq 0\}$, for which the feasible region is depicted in the first subfigure of Figure \ref{fig:ActConst}. Suppose the constraints are indexed by $\mathcal{K} = \{1, 2, 3, 4, 5\}$ and the set $\Pi$ allows $2.5 \leq b_1 \leq 4 $, keeping all other parameters fixed. For simplicity, we assume that the noisy data obtained is such that the correct vertex can always be recovered by solving \eqref{eqn:P1} for each of the experiments individually. Therefore, $\widehat{\mathcal{C}}_1 = \mathrm{cone}(\{(-1, 0), (2, 3)\})$. Now, for the next experiment to result in a cone with $\eta > 0$, we have to choose $b_1 > 3$, which will yield a different set of active constraints. For any other choice of $b_1$, $\widehat{\mathcal{C}}_2$ will be the same as $\widehat{\mathcal{C}}_1$. In Figure \ref{fig:ActConst}, we show the case when $b_1$ for the second experiment is chosen to be 4. This $b_1$ results in a $\widehat{\mathcal{C}}_2 = \mathrm{cone}(\{(0, 1), (2, 3)\})$, thereby effectively reducing the size of the admissible set and increasing the confidence in the point estimate selected from this set. 
\end{example}

\begin{figure}[ht]
    \includegraphics[clip, trim=0 0 0 0, width=\textwidth]{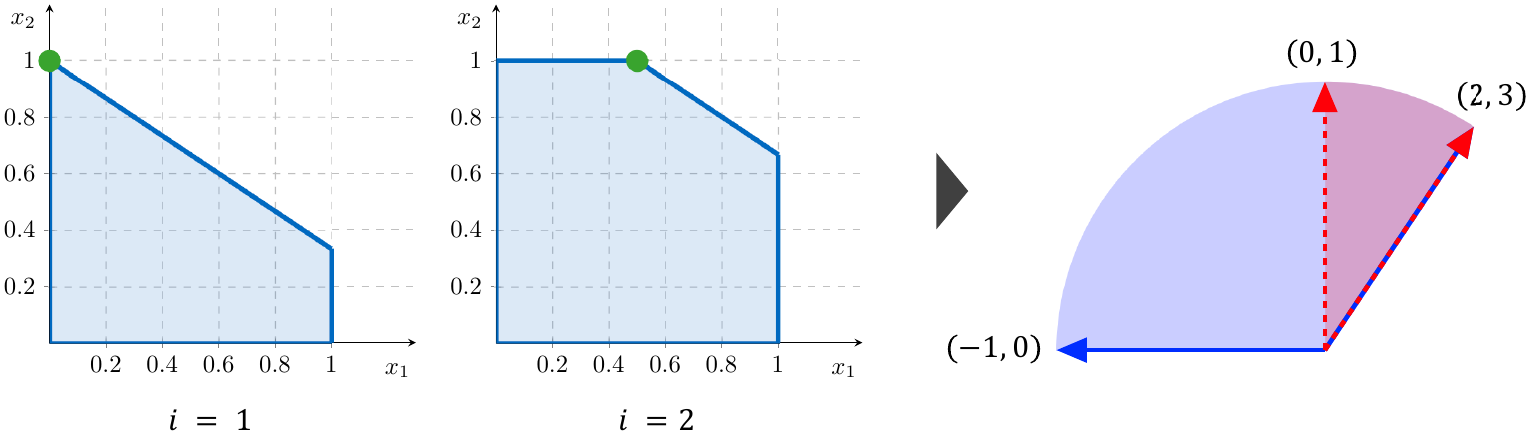}
    \centering
    \caption{Illustration of the problem described in Example \ref{examp:varyingb}. The vectors are normalized to the same length. Filled green circles depict the denoised estimates of the noisy observations (omitted).}
    \label{fig:ActConst}
\end{figure}

\section{Additional Results for the Case Study from Section \ref{sec:CustPref}} \label{sec:CustPrefAddRes}
In this section, we present some additional results for the customer preference learning case study considered in Section \ref{sec:CustPref}. Specifically, in Figure \ref{fig:AddlResults}, we show for each case the evolution of the prediction error for each of the ten random instances of \eqref{eqn:IOP}. Here, one can notice that in case of a failed run with the decomposition algorithm, the solution method is still able to process more than 50\% of the experiments. Unlike Algorithm \ref{alg:Two-phase}, a failed run with Algorithm \ref{alg:Online Algorithm} yields a partial estimate which can be used to generate predictions, albeit with a slightly higher prediction error compared to the mean values observed after the completion of 100 experiments. Further, Figure \ref{fig:AddlResults} shows that in the majority of the cases, the prediction error shows a continuous decrease as data from more experiments are added; the spikes observed in the mean curves are a result of just 1 or 2 instances showing a momentary increase in the prediction error. This shows that our IOP formulation is fairly robust in generating accurate predictions even in high-noise situations.

\begin{figure}[h!]
\centering
\begin{minipage}{0.33\linewidth}
\resizebox{\linewidth}{!}{\includegraphics{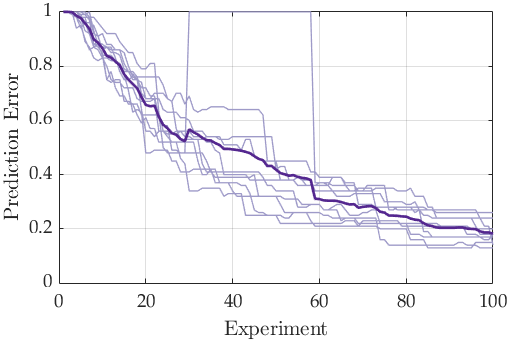}}%
\caption*{(a) $n = 25, \sigma = 0.01, J = 5$}%
\end{minipage}%
\begin{minipage}{0.33\linewidth}
\resizebox{\linewidth}{!}{\includegraphics{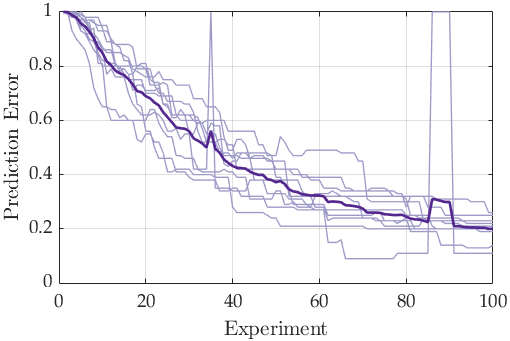}}%
\caption*{(b) $n = 25, \sigma = 0.1, J = 20$}%
\end{minipage}%
\begin{minipage}{0.33\linewidth}
\resizebox{\linewidth}{!}{\includegraphics{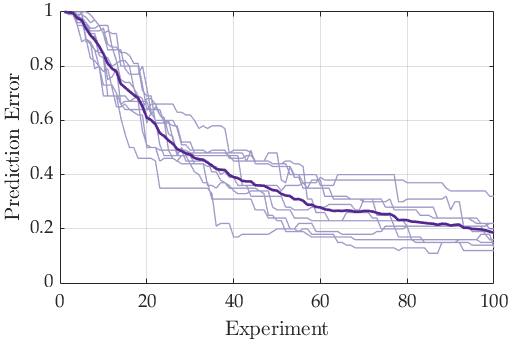}}%
\caption*{(c) $n = 25, \sigma = 0.1, J = 250$}%
\end{minipage}%
\vspace{2em}
\noindent
\begin{minipage}{0.33\linewidth}
\resizebox{\linewidth}{!}{\includegraphics{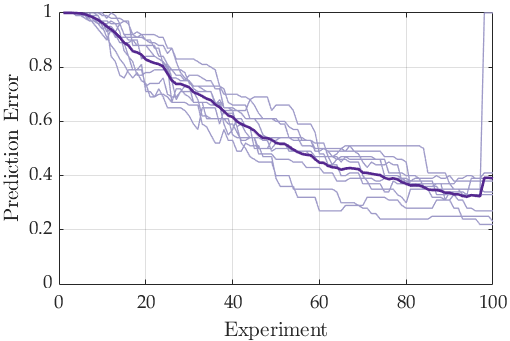}}%
\caption*{(d) $n = 50, \sigma = 0.01, J = 5$}%
\end{minipage}%
\begin{minipage}{0.33\linewidth}
\resizebox{\linewidth}{!}{\includegraphics{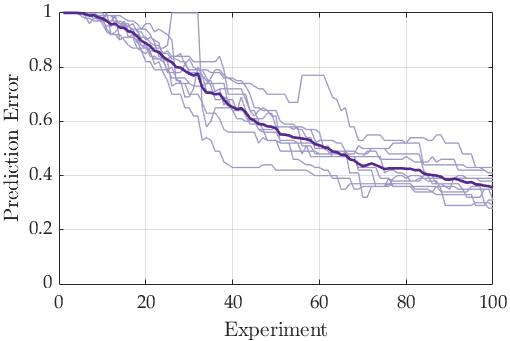}}%
\caption*{(e) $n = 50, \sigma = 0.1, J = 20$}%
\end{minipage}%
\begin{minipage}{0.33\linewidth}
\resizebox{\linewidth}{!}{\includegraphics{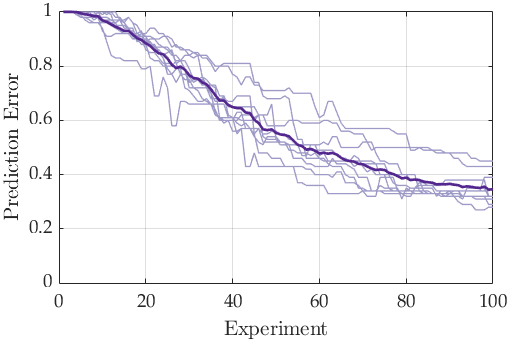}}%
\caption*{(f) $n = 50, \sigma = 0.1, J = 250$}%
\end{minipage}%
\vspace{2em}
\noindent
\begin{minipage}{0.33\linewidth}
\resizebox{\linewidth}{!}{\includegraphics{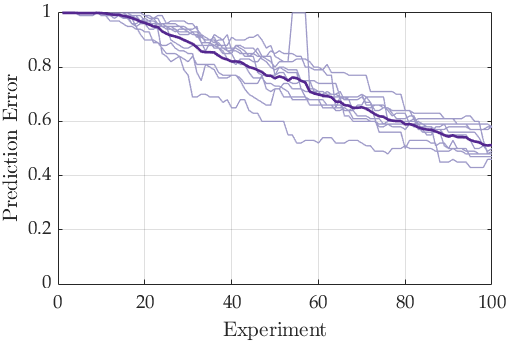}}%
\caption*{(g) $n = 100, \sigma = 0.01, J = 5$}%
\end{minipage}%
\begin{minipage}{0.33\linewidth}
\resizebox{\linewidth}{!}{\includegraphics{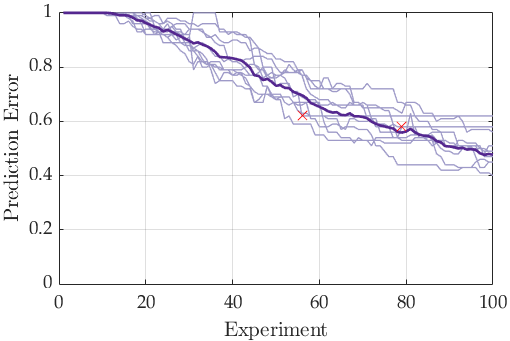}}%
\caption*{(h) $n = 100, \sigma = 0.1, J = 20$}%
\end{minipage}%
\begin{minipage}{0.33\linewidth}
\resizebox{\linewidth}{!}{\includegraphics{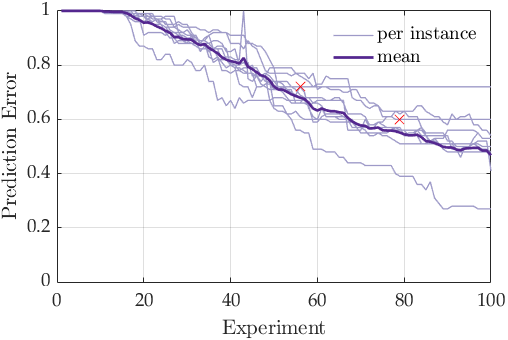}}%
\caption*{(i) $n = 100, \sigma = 0.1, J = 250$}%
\end{minipage}%
\caption{Change in prediction error as experiments are added. Thick lines show mean values across all solved instances, thinner lines show the change in prediction error for all the ten instances. Red cross markers highlight the points at which Algorithm \ref{alg:Online Algorithm_1} timed out for a specific instance.}
\label{fig:AddlResults}
\end{figure}

\section{Data for Production Planning Case Study from \\ Section \ref{sec:ProdPlan}} \label{sec:ProdPlanData}

We have $q_m^{\mathrm{min}} = 0, q_m^{\mathrm{max}} = 200, x_{pmh}^{\mathrm{max}} = 300$ for all $p \in \mathcal{P}, m \in \mathcal{M}, h \in \mathcal{H},$ and \\
\begin{equation*}
    w_{mh}^{\mathrm{max}} = \Bigg\{\begin{array}{lr}
        300, \; \text{if } m \in \{1, \ldots, 9 \} \\
        0, \; \mathrm{otherwise} \\
        \end{array} \forall \; h \in \mathcal{H}.
\end{equation*}
We assume that the nominal value of demand ${d}_{mh}$ remains the same across all time periods, i.e. nominal ${d}_{mh} = D_m \; \forall h \in \mathcal{H}$. Table \ref{tab:ProdPlanPara1} lists the values of $D_m, q_m^0$ used for all products. Materials not listed are raw materials, byproducts, or intermediate products for which $D_m$ and $q^0_m$ are set to zero. Table \ref{tab:ProdPlanPara2} shows the nominal $\mu_{pm}$ parameters.

\begin{table}[h!]
	\centering
	\begin{tabular}{@{}ccccccccccccccccc@{}}
		\toprule
		\multicolumn{17}{c}{Materials} \\
		& 10 & 11 & 12 & 13 & 14 & 15 & 16 & 17 & 18 & 19 & 20 & 21 & 22 & 23 & 24 & 25 \\
		\midrule
		$D_m$ & 40 & 150 & 30 & 75 & 60 & 30 & 30 & 50 & 150 & 70 & 30 & 75 & 100 & 75 & 250 & 50 \\
		$q_m^0$ & 10 & 5 & 0 & 5 & 5 & 7 & 10 & 2 & 3 & 5 & 10 & 7 & 5 & 5 & 3 & 3 \\
		\bottomrule 
	\end{tabular}
	\caption{\label{tab:ProdPlanPara1}Values of $D_m$ and $q_m^0$ parameters used for the case study presented in Section \ref{sec:ProdPlan}.}
\end{table}

\begin{landscape}
\begin{table}[h!]
\resizebox{1.3\textwidth}{!}{%
\begin{tabular}{@{}cccccccccccccccccccccccccccccc@{}}
    \toprule
     &  & \multicolumn{28}{c}{Materials} \\
     &  & 1 & 2 & 3 & 4 & 5 & 6 & 7 & 8 & 9 & 10 & 11 & 12 & 13 & 14 & 15 & 16 & 17 & 18 & 19 & 20 & 21 & 22 & 23 & 24 & 25 & 26 & 27 & 28 \\ 
     \midrule
    \parbox[t]{2mm}{\multirow{38}{*}{\rotatebox[origin=c]{90}{Processes}}} & 1 & -0.58 &  &  &  &  & -0.63 &  &  &  &  & 1 &  &  &  &  &  &  &  &  &  &  &  &  &  &  &  &  &  \\
     & 2 &  &  &  &  &  & -0.64 &  &  &  &  &  & 1 &  &  &  &  &  &  &  &  &  &  &  &  &  &  &  &  \\
     & 3 & -0.055 & -1.25 &  &  &  &  &  &  &  &  & 1 &  &  &  &  &  &  &  &  &  &  &  &  &  &  &  &  &  \\
     & 4 &  & -0.4 & -0.69 &  &  &  &  &  &  &  &  &  &  & 1 &  &  &  &  &  &  &  &  &  &  &  &  &  &  \\
     & 5 &  &  &  &  &  &  &  &  &  &  &  &  & 1 & -2.3 &  &  & 1.7 &  &  &  &  &  &  &  &  &  &  &  \\
     & 6 &  & -0.74 &  &  &  &  &  &  &  &  &  &  &  &  & 1 &  &  &  &  &  &  &  &  &  &  &  &  &  \\
     & 7 &  &  &  &  &  &  &  &  &  &  &  &  & 1 &  & -1.1 &  &  &  &  &  &  &  &  &  &  &  &  &  \\
     & 8 &  &  & -1 &  &  &  &  &  &  &  &  &  &  &  &  & 1 &  &  &  &  &  &  &  &  &  &  &  &  \\
     & 9 &  &  &  &  &  &  &  &  &  &  &  &  &  &  &  & -1.26 & 1 &  &  &  &  &  &  &  &  &  &  &  \\
     & 10 &  &  &  &  &  &  &  &  &  &  &  &  & -1.57 &  &  &  &  &  &  &  &  &  &  &  &  &  & 1 &  \\
     & 11 &  &  & -1.01 &  &  &  &  &  &  &  &  &  &  &  &  &  & 1 &  &  &  &  &  &  &  &  &  &  &  \\
     & 12 &  &  & -0.76 & -0.28 &  &  &  & 1 &  &  &  &  &  &  &  &  &  &  &  &  &  &  &  &  &  &  &  &  \\
     & 13 &  &  &  &  &  &  &  & -1.14 &  &  &  &  &  &  &  &  &  & 1 &  &  &  &  &  &  &  &  &  &  \\
     & 14 &  & -0.78 &  &  &  &  &  &  &  &  &  &  & 1 &  &  &  &  &  &  &  &  &  &  &  &  &  &  &  \\
     & 15 &  &  &  &  &  &  &  &  &  &  &  & 1 &  &  &  &  &  &  & -1.34 &  &  &  &  &  &  &  &  &  \\
     & 16 &  &  &  & -0.6 &  &  &  &  &  &  &  &  &  &  &  &  &  &  & 1 &  &  &  &  &  &  &  &  &  \\
     & 17 &  &  &  & -0.67 &  &  &  &  &  &  &  & 1 &  &  &  &  &  &  &  &  &  &  &  &  &  &  &  &  \\
     & 18 &  &  &  &  &  &  &  &  &  &  &  & -1.1 &  &  &  &  &  &  &  & 1 &  &  &  &  &  &  &  &  \\
     & 19 &  &  &  &  &  &  &  &  &  &  &  &  &  &  &  &  &  &  & -0.98 & 1 &  &  &  &  &  &  &  &  \\
     & 20 &  &  &  & -0.35 &  &  &  &  &  &  &  &  &  &  &  &  &  &  &  & -0.71 & 1 &  &  &  &  &  &  &  \\
     & 21 &  &  &  &  &  & -0.32 &  &  &  &  &  &  &  &  &  &  &  &  &  & -0.72 & 1 &  &  &  &  &  &  &  \\
     & 22 &  &  &  & -0.88 & 1 &  &  &  &  &  &  &  &  &  &  &  &  &  &  &  &  &  &  & 0.03 &  &  &  &  \\
     & 23 & -0.56 &  &  &  & -0.92 &  &  &  &  &  & 1 &  &  &  &  &  &  &  &  &  &  &  &  &  &  &  &  &  \\
     & 24 &  &  &  & -0.39 &  &  &  &  &  &  &  &  &  &  &  &  &  &  &  &  &  &  &  &  &  &  &  & 1 \\
     & 25 &  &  &  &  & 1 &  &  &  &  &  &  &  &  &  &  &  &  &  &  &  &  &  &  &  &  &  &  & 1 \\
     & 26 &  &  &  & -0.3 &  &  &  &  &  &  &  &  &  &  &  &  &  &  &  &  &  &  &  &  &  & 1 &  &  \\
     & 27 &  &  &  &  &  &  &  &  &  &  &  &  &  &  &  &  &  &  &  & -0.65 &  & 1 &  &  &  &  & -0.46 &  \\
     & 28 &  &  &  &  &  &  & -0.56 &  &  & -0.56 &  &  &  &  &  &  &  &  &  & 1 &  &  &  &  &  &  &  &  \\
     & 29 &  &  &  &  &  &  &  &  &  &  &  & -1.2 &  &  &  &  &  &  &  &  &  & 1 &  &  &  &  &  &  \\
     & 30 &  &  &  & -1.17 &  &  &  &  &  &  &  &  &  &  &  &  &  &  &  &  &  &  &  &  & 1 &  &  &  \\
     & 31 &  &  &  &  & -0.75 &  &  &  &  &  &  &  &  &  &  &  &  &  &  &  &  &  &  & 1 &  &  &  &  \\
     & 32 &  &  &  & -0.53 &  &  &  &  &  &  &  &  &  &  &  &  &  &  &  &  &  &  &  & 1 &  &  &  &  \\
     & 33 &  &  &  &  &  &  &  &  &  &  &  & -0.6 &  &  &  &  &  &  &  & -0.82 & 1 &  &  &  &  &  &  &  \\
     & 34 &  &  &  &  &  &  &  &  &  & -0.42 &  &  &  &  &  &  &  &  &  &  &  &  &  &  & 1 &  &  &  \\
     & 35 &  &  &  &  &  &  &  &  & -0.5 & 1 &  &  &  &  &  &  &  &  &  &  &  &  &  &  &  &  &  &  \\
     & 36 &  &  &  &  &  &  & -0.53 &  &  &  &  &  &  &  &  &  &  &  &  &  &  &  &  & 1 & -0.57 &  &  &  \\
     & 37 &  &  &  &  &  &  &  &  &  &  &  &  &  &  &  &  &  &  &  &  &  &  &  & 1 &  &  &  & -1.44 \\
     & 38 &  & 0.38 & 0.22 & 1 &  &  &  &  & -3.08 &  &  &  &  &  &  &  &  &  &  &  &  &  &  &  &  & 1.81 &  &  \\ \bottomrule
\end{tabular}}
\caption{\label{tab:ProdPlanPara2}Nominal values of the conversion factors $\mu_{pm}.$}
\end{table}
\end{landscape}

\end{document}